\documentclass[final,leqno,letterpaper]{etna}

\usepackage{mathtools}
\usepackage{amssymb}
\usepackage{amsmath}
\usepackage{amsfonts}

\usepackage{algorithm}

\usepackage{cleveref}
\usepackage{subcaption}

\usepackage{lipsum}
\usepackage{graphicx}
\usepackage{epstopdf}
\usepackage{algorithmic}
\usepackage{listings}

\usepackage{booktabs}
\usepackage{ragged2e}
\usepackage{array}
\usepackage{diagbox}

\usepackage{microtype}

\usepackage{lineno}

\hypersetup{
  pdftitle={An algebraic multiscale framework for two-level overlapping Schwarz preconditioners},
  pdfauthor={Filipe Cumaru, Hadi Hajibeygi, Cornelis Vuik, Alexander Heinlein},
  pdfkeywords={multiscale methods, overlapping Schwarz preconditioning, domain decomposition methods, heterogeneous problems}
}

\crefname{lstlisting}{listing}{listings}
\Crefname{lstlisting}{Listing}{Listings}

\def\bbbr{{\rm I\!R}} 
\DeclareMathOperator{\dist}{dist} 
\DeclareMathOperator{\tr}{tr} 
\DeclareMathOperator{\diam}{diam} 

\title{An algebraic multiscale framework for two-level overlapping Schwarz preconditioners}

\author{
  Filipe Cumaru\thanks{Delft Institute of Applied Mathematics, Delft University of Technology, 2628 CD Delft, The Netherlands (\tt{\{f.a.cumarusilvaalves@tudelft.nl, a.heinlein@tudelft.nl, c.vuik\}@tudelft.nl})}
  \and Hadi Hajibeygi\thanks{Department of Geoscience and Engineering, Delft University of Technology, 2628 CV Delft, The Netherlands (\tt{h.hajibeygi@tudelft.nl})}
  \and Cornelis Vuik\footnotemark[1]
  \and Alexander Heinlein\footnotemark[1]
}
\shorttitle{Algebraic multiscale coarse space for OAS}
\shortauthor{F. Cumaru, H. Hajibeygi, C. Vuik, and A. Heinlein}

\begin{document}

\maketitle

\begin{abstract}
The robustness of two-level overlapping additive Schwarz (OAS) preconditioners for highly heterogeneous elliptic problems hinges on the choice of the coarse space. We propose an algebraic multiscale (AM) coarse space built from the system matrix alone through an algebraic formulation of the multiscale finite element method (MsFEM). The coarse basis functions are energy-minimizing extensions of the solutions of reduced boundary condition (RBC) problems on the subdomain interface, so that coefficient variations along the interface enter the coarse space. By posing the RBC problems on slabs that extend the interface, the construction also applies to general unstructured domain decompositions. Embedding the AM formulation in an energy-minimizing framework extends it to linear elasticity, and we show that existing condition number bounds for OAS with MsFEM coarse spaces carry over to the algebraic setting. For problems with high coefficient inclusions on the coarse nodes, the AM iteration count and condition number are unaffected by increasing the contrast. In 3D scalar diffusion with log-normal coefficients, AM needs about as many iterations as the generalized Dryja--Smith--Widlund (GDSW) coarse space, while solving a coarse problem ten times smaller. For linear elasticity, omitting the coarse basis functions associated with the rotational rigid body modes yields a fully algebraic coarse space that still scales well for homogeneous problems.
\end{abstract}



\section{Introduction} \label{sec:intro}

Large heterogeneous problems arise in various contexts such as fluid flow in porous media \cite{SLEIPNER2019, LI2022BILLIONCELL, SPE112024} and deformation of different materials \cite{NAKAJIMA2019TIREMECHANICS, TASAN2015STEEL}. The numerical simulation of these problems involves the solution of large systems of linear equations. Coupled with the multiscale nature of the physical properties of the media or materials involved, this task quickly becomes computationally demanding and requires a good choice of linear solver for an accurate and efficient solution.

Due to the superlinear memory and time complexity of direct solvers, iterative methods such as Krylov subspace methods are usually employed for the solution of large systems of linear equations. However, the convergence of these methods depends on the spectral properties of the system matrix \cite{SAAD2003}. An effective preconditioner is needed to accelerate the solution of the problem.

Domain decomposition methods provide a parallel divide-and-conquer preconditioning framework for the solution of large elliptic problems. Among the available formulations \cite{DOHRMANN2003BDDC, CHARBEL2001FETIDP, CHARBEL1991FETI}, we turn our attention to overlapping additive Schwarz (OAS) methods. These methods are based on the solution of local problems defined on overlapping subdomains. However, the classical one-level OAS method is generally not scalable, as the conditioning of the problem deteriorates when more subdomains are used \cite{CHANMATHEW1994DD}. The two-level variant of the OAS method is able to reinforce the scalability property via the solution of a global coarse problem that couples the local problems.

The robustness of two-level OAS methods depends on the global coarse problem, more specifically, the coarse space spanned by a set of coarse basis functions. Originally, a set of finite element shape functions defined on a coarse triangulation would be employed \cite{TOSELLI2005DDBOOK}, which require geometric information. For a purely algebraic construction, the basis functions may be, for instance, computed based on the non-overlapping domain decomposition. Coarse spaces such as the generalized Dryja--Smith--Widlund (GDSW) \cite{DOHRMANN2008GDSW1,DOHRMANN2008GDSW2} and the reduced dimension GDSW (RGDSW) \cite{DOHRMANN2017RGDSW} define the coarse basis functions as energy-minimizing extensions of trace functions defined on the interface of the non-overlapping subdomains that span the null space of the problem. Coarse spaces based on the solution of local generalized eigenvalue problems have been proposed in \cite{AGULLO2019,ALDAAS2023,DOLEAN2012DTN,EFENDIEV2013GMSFEM,GALVISEFENDIEV2010,GANDERLONELAND2017SHEM,HEINLEIN2018MS,HEINLEIN2019AGDSW,HEILEINSMETANA2025,SPILLANE2014GENEO}. In addition, related work on spectral multigrid methods has been presented in \cite{BREZINA2006ADAPTIVEAMG,CHARTIER2003RHOAMGE,EFENDIEV2011SPECTRALAMG}.

The multiscale finite element method (MsFEM) \cite{HOU1997MSFEM} introduces a set of basis functions computed as energy-minimizing extensions, into the interior of the coarse elements, of trace functions that are the solution of reduced boundary condition (RBC) problems on the interface. This procedure can better capture changes in the physical properties of the medium. Other variants have been proposed, such as the multiscale finite volume method (MsFV) \cite{JENNY2003MSFV} and multiscale mixed finite element methods \cite{ARBOGAST2002MMFEM,ARBOGASTBRYANT2002MMFEM}. Coarse spaces based on the MsFEM have been studied in \cite{BUCK2013MSFEMLE, GRAHAM2007MSPDES, HEINLEIN2018MS}. An algebraic formulation of MsFEM was proposed in \cite{LUNATI2009OPMSFV} and further developed as the algebraic multiscale solver (AMS) \cite{SOUZA2022AMSU,WANG2014AMS,ZHOUTCHELEPI2012TAMS} for highly heterogeneous porous media. AMS is employed as a preconditioning scheme for Krylov subspace methods in which the algebraic MsFEM functions are coupled with a local stage, e.g., ILU(0) or BILU, to improve convergence by better capturing low-frequency error modes. The performance of AMS was found to be comparable to that of an algebraic multigrid method.

In this work, we introduce a two-level OAS preconditioner with an algebraic multiscale coarse space for highly heterogeneous elliptic problems. We also discuss how it can be related to the RGDSW method \cite{DOHRMANN2017RGDSW}, allowing us to extend the original algebraic multiscale formulation to the linear elasticity equations. In addition, by posing the RBC problems on slabs that extend the interface entities, the construction also applies to general unstructured domain decompositions whose interface is not algebraically connected. Moreover, we review the bound on the condition number of the Schwarz operator with multiscale coarse spaces developed in \cite{GRAHAM2007MSPDES} when the algebraic formulation is considered. Although the MsFEM has been previously employed as a coarse space for the Schwarz method \cite{AARNES2002, BUCK2013MSFEMLE, GALVISEFENDIEV2010, HEINLEIN2018MS}, to the best of our knowledge, this is the first time a purely algebraic approach has been used in this context besides our own initial study in \cite{CUMARU2026DD28}. In \cite{ZHOU2026ALGEBRAICMS}, an algebraic formulation of the Generalized MsFEM (GMsFEM) \cite{EFENDIEV2013GMSFEM} was proposed, which is based on the solution of generalized eigenvalue problems at the subdomain level. We refrain from solving such problems and instead focus on the classical MsFEM based on the reduced boundary condition problems. Although robustness cannot be guaranteed in general, the resulting preconditioning scheme can handle a wide range of heterogeneous elliptic problems.

The remainder of this paper is organized as follows. In \cref{sec:model_problem} the formulations of the scalar diffusion and the linear elasticity equations are reviewed. Next, in \cref{sec:two_level_oas}, we introduce the two-level OAS method as well as energy-minimizing coarse spaces. In \cref{sec:ams_coarse_space}, we discuss the MsFEM coarse space and its algebraic formulation, followed by a computational complexity analysis and a convergence analysis for the 2D scalar diffusion problem. We then describe how slab interface entities handle algebraically unconnected interfaces and propose a framework shared with RGDSW for the design of algebraic multiscale coarse spaces. In \cref{sec:numerical_results}, we present numerical results on the robustness and weak scalability of the proposed methods, as well as a comparison with the AMS preconditioner. Finally, we draw conclusions in \cref{sec:conclusions}.

\section{Model problems} \label{sec:model_problem}

In this section, we describe the formulation of the model problems used in the study of the preconditioning methods discussed in this work.

\subsection{Scalar diffusion}

As the first model problem, we consider the scalar diffusion problem of the form
\begin{alignat}{2}
    \label{eq:diffusion_strong_form}
    - \nabla \cdot \left( \alpha (x) \nabla u (x) \right) &= f(x) \quad \quad && \text{in} \, \Omega \subset \bbbr^d, \\ \nonumber
    u &= 0 \quad \quad && \text{on} \, \partial \Omega_D,
\end{alignat}
where $d = 2,\, 3$ and $\partial \Omega_D = \partial \Omega$ is the Dirichlet boundary of the domain. The scalar coefficient function $\alpha : \Omega \to \bbbr$ is such that $\alpha_{\text{min}} < \alpha < \alpha_{\text{max}} < \infty$ for $\alpha_{\text{min}},\, \alpha_{\text{max}} \in \bbbr^{+}$. Furthermore, $\alpha$ is highly heterogeneous, possibly containing large coefficient jumps.

The weak form of \cref{eq:diffusion_strong_form} can be described as: find $u \in V$ such that
\begin{equation} \label{eq:diffusion_weak_form}
    a(u, v) = L(v), \quad \forall v \in V,
\end{equation}
where the symmetric positive definite bilinear form $a(\cdot, \cdot): V \times V \to \bbbr$ and the linear form $L(\cdot): V \to \bbbr$ are, respectively,
\begin{align*}
    a(u, v) &= \int_{\Omega} \alpha (x) \left( \nabla u \right)^{\top} \nabla v \; \text{d}x, \\
    L(v) &= \int_{\Omega} f(x) v \; \text{d}x, \quad f \in L^2(\Omega).
\end{align*}
Here, we consider the function space $V = H_0^1(\Omega)$.

The problem in \cref{eq:diffusion_weak_form} is discretized using piecewise linear finite elements on a triangulation $\tau_h$ of the computational domain $\Omega$. This results in the system of linear equations
\begin{equation} \label{eq:les}
    A\mathbf{u} = \mathbf{b},
\end{equation}
where $A$ is the sparse symmetric positive definite stiffness matrix, $\mathbf{b}$ is the right-hand side vector, and $\mathbf{u}$ is the vector of unknowns corresponding to the finite element solution in the finite element space $V^h(\Omega)$. The system of equations is solved using a Krylov subspace method. Due to the heterogeneity of the coefficient $\alpha(x)$, the system in \cref{eq:les} is usually ill-conditioned and demands a preconditioner for an efficient solution using an iterative method.

\subsection{Linear elasticity}

Additionally, we investigate the equations of compressible linear elasticity for an isotropic material:
\begin{alignat}{2} \label{eq:le_strong_form}
    - \nabla \cdot \sigma \left( \mathbf{u}(x) \right) &= \mathbf{f}(x) \quad \quad && \text{in } \, \Omega \subset \bbbr^d, \\ \nonumber
    \mathbf{u} &= \mathbf{0} \quad \quad && \text{on } \, \partial \Omega_D,
\end{alignat}
where $d = 2,\, 3$, $\partial \Omega_D = \partial \Omega$ is the Dirichlet boundary of the domain, $\mathbf{u} : \bbbr^d \to \bbbr^d$ is the displacement field, $\mathbf{f} : \bbbr^d \to \bbbr^d$ is the body force, and $\sigma(\mathbf{u})$ is the linearized stress tensor defined as
\begin{equation*}
    \sigma(\mathbf{u}) \coloneq 2 \mu \epsilon (\mathbf{u}) + \lambda \tr{\left( \epsilon (\mathbf{u}) \right)} I.
\end{equation*}
The linearized strain tensor $\epsilon(\mathbf{u})$ is given by
\begin{equation*}
    \epsilon(\mathbf{u}) \coloneq \frac{1}{2} \left( \nabla \mathbf{u} + (\nabla \mathbf{u})^{\top} \right).
\end{equation*}
In addition, the Lamé parameters $\lambda(x), \, \mu(x): \bbbr^d \to \bbbr$ are defined in terms of the Young modulus $E: \Omega \to \bbbr^+$ and the Poisson ratio $\nu \in [0, 1/2)$ as
\begin{equation*}
    \lambda(x) \coloneq \frac{E \nu}{(1 + \nu)(1 - 2\nu)}, \quad \mu(x) \coloneq \frac{E}{2(1 + \nu)}.
\end{equation*}
Like the scalar coefficient in \cref{eq:diffusion_strong_form}, the Young modulus may also be highly heterogeneous and contain large jumps.

The weak form of \cref{eq:le_strong_form} can be described in the same way as \cref{eq:diffusion_weak_form} with $V = \left( H_0^1(\Omega) \right)^d$, and the bilinear and linear forms defined as
\begin{equation} \label{eq:le_weak_form}
    \begin{aligned}
        a(\mathbf{u}, \mathbf{v}) & \coloneq \int_{\Omega} 2 \mu (x) \left( \epsilon (\mathbf{u}(x)) : \epsilon (\mathbf{v}(x)) \right) \, \mathrm{d}x + \int_{\Omega} \lambda (x) \left( \nabla \cdot \mathbf{u}(x) \right) \left( \nabla \cdot \mathbf{v}(x) \right) \, \mathrm{d}x, \\
        L(\mathbf{v}) & \coloneq \int_{\Omega} \mathbf{f}(x) \cdot \mathbf{v}(x) \, \mathrm{d}x.
    \end{aligned}
\end{equation}

As with the diffusion problem, we apply piecewise linear finite elements to discretize \cref{eq:le_weak_form}, which results in a system of linear equations of the form in \cref{eq:les}.

\section{Two-level overlapping Schwarz methods} \label{sec:two_level_oas}

\begin{figure}[t]
    \centering
    \begin{subfigure}[b]{0.3\textwidth}
        \includegraphics[height=0.89\textwidth]{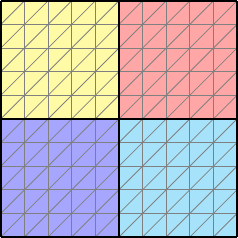}
        \caption{}
        \label{fig:non_ovlp_dd}
    \end{subfigure}
    \hfil
    \begin{subfigure}[b]{0.3\textwidth}
        \includegraphics[height=0.89\textwidth]{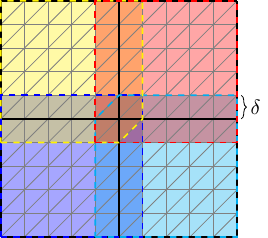}
        \caption{}
        \label{fig:ovlp_dd}
    \end{subfigure}
    \caption{An example of a non-overlapping (a) and an overlapping (b) domain decomposition used in the two-level Schwarz method. The dashed regions in (b) indicate the overlapping subdomains. Adapted from \cite{HEINLEIN2022LANDICE}.}
    \label{fig:dd_example}
\end{figure}

We examine the two-level overlapping additive Schwarz method as a preconditioning scheme for \cref{eq:les}. Let the computational domain $\Omega$ be partitioned into $N$ non-overlapping subdomains $\left\{ \Omega_i \right\}_{i = 1}^N$. Each non-overlapping subdomain is extended by layers of mesh elements to obtain an overlapping decomposition $\left\{ \Omega_i' \right\}_{i = 1}^N$ with overlap width $\delta$. This is illustrated in \cref{fig:dd_example}. The two-level overlapping additive Schwarz (OAS-2) preconditioner in matrix form is given by
\begin{equation} \label{eq:two_level_oas_preconditioner}
    M_{\text{OAS-2}}^{-1} = \Phi A_0^{-1} \Phi^{\top} + \sum_{i = 1}^{N} R_i^{\top} A_i^{-1} R_i,
\end{equation}
where $R_i: V \to V_i \coloneq V^h(\Omega_i')$ is the restriction to the local finite element space on the overlapping subdomain $\Omega_i'$, $A_i = R_i A R_i^{\top}$ are the corresponding local subdomain matrices, and $A_0 = \Phi^{\top} A \Phi$ is the coarse problem matrix.

The coarse problem is defined on the coarse space spanned by the coarse basis functions that form the prolongation operator $\Phi$. In general, this additional problem is required for robustness and scalability \cite{TOSELLI2005DDBOOK}. The choice of coarse space is a crucial ingredient of the method, as it determines the robustness of the preconditioner for different problems and, more specifically, with respect to the contrast of the coefficient function.

\subsection{Energy-minimizing coarse spaces} \label{sec:energy_min_coarse_spaces}

\begin{figure}[t]
    \centering
    \begin{subfigure}[b]{0.3\textwidth}
        \includegraphics[width=\textwidth]{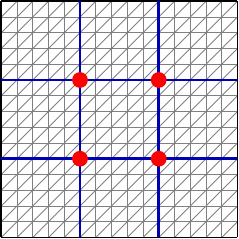}
        \caption{}
        \label{fig:dd_interface_2d}
    \end{subfigure}
    \hfil
    \begin{subfigure}[b]{0.3\textwidth}
        \includegraphics[width=\textwidth]{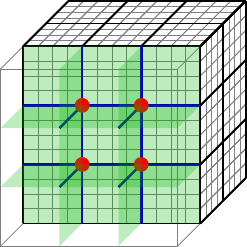}
        \caption{}
        \label{fig:dd_interface_3d}
    \end{subfigure}
    \caption{Interface of a non-overlapping domain decomposition in two (a) and three (b) dimensions. The vertex, edge and face components are highlighted in red, blue and green, respectively.}
    \label{fig:dd_interface}
\end{figure}

The coarse space of the two-level method may be computed as energy-minimizing extensions of trace functions defined on the interface of the subdomains. We start by defining $\Gamma \coloneq \bigcup_{i = 1}^{N} \partial \Omega_i \setminus \partial \Omega_D$ and $\text{I} \coloneq \Omega \setminus \Gamma$ as the set of interface and interior nodes, respectively. The non-Dirichlet boundary nodes are considered as part of the interior set. We note that the interface and its components may be defined algebraically; cf. \cite{HEINLEIN2020FROSCH}. We may reorder the system matrix $A$ and the prolongation operator $\Phi$ such that
\begin{equation*}
    A = \begin{bmatrix}
        A_{\text{I}\text{I}} & A_{\text{I}\Gamma} \\
        A_{\Gamma\text{I}} & A_{\Gamma\Gamma}
    \end{bmatrix}\, \text{and }
    \Phi = \begin{bmatrix}
        \Phi_{\text{I}} \\
        \Phi_{\Gamma}
    \end{bmatrix}.
\end{equation*}
Given the value of the coarse basis functions on the interface $\Phi_{\Gamma}$, the energy-minimizing extension into the interior is computed as
\begin{equation} \label{eq:energy_min_ext}
    \Phi = \begin{bmatrix}
        -A_{\text{I}\text{I}}^{-1} A_{\text{I}\Gamma} \\
        I_{\Gamma\Gamma}
    \end{bmatrix}
    \Phi_{\Gamma},
\end{equation}
where $I_{\Gamma\Gamma}$ is the identity matrix of dimension $|\Gamma| \times |\Gamma|$. We note that $A_{\text{I}\text{I}} = \diag_{i = 1}^N \left( A_{\text{I}\text{I}}^{(i)} \right)$ is block diagonal and therefore the application of $A_{\text{I}\text{I}}^{-1}$ can be done in parallel.

Different coarse spaces can be defined depending on $\Phi_{\Gamma}$ in \cref{eq:energy_min_ext}. In the following subsections, we turn our attention to two formulations for the interface trace values: the generalized Dryja--Smith--Widlund (GDSW) \cite{DOHRMANN2008GDSW1,DOHRMANN2008GDSW2} and the reduced dimension GDSW (RGDSW) \cite{DOHRMANN2017RGDSW}. These methods are designed such that the null space of the problem is included in the coarse space. This is the \textit{null space property}, and it is necessary to ensure the scalability of the preconditioner. It ensures that the convergence rate does not depend explicitly on the number of subdomains; cf. \cite[subsection~4.3.4]{SMITHBJORSTADGROPP1996DDBOOK}.

\subsubsection{The GDSW coarse space} \label{sec:gdsw_coarse_space}

In the GDSW coarse space, the interface $\Gamma$ is decomposed into connected components that are shared by the same set of subdomains. This amounts to splitting the interface into vertices, edges and, in 3D, faces. An example of such a decomposition is shown in \cref{fig:dd_interface}. The values of the coarse basis functions on the interface are set as the restriction of the null space of the problem to each interface component $\Gamma_i$. The interface operator can thus be written as
\begin{equation*}
    \Phi_{\Gamma}^{\text{GDSW}} =
    \begin{bmatrix}
        R_{\Gamma_1}^{\top} \Phi_{\Gamma_1} & \dots & R_{\Gamma_n}^{\top} \Phi_{\Gamma_n}
    \end{bmatrix} =
    \begin{bmatrix}
        R_{\Gamma_1}^{\top} Z_{\Gamma_1} & \dots & R_{\Gamma_n}^{\top} Z_{\Gamma_n}
    \end{bmatrix},
\end{equation*}
where $Z_{\Gamma_i}$ is the restriction of the null space to the interface component $\Gamma_i$. In general, the components $R_{\Gamma_i}^{\top} Z_{\Gamma_i}$ might be linearly dependent, thus requiring an additional orthonormalization step. For the scalar diffusion problem, the null space is spanned by
\begin{equation*}
    r_u \coloneq \begin{bmatrix}
        1
    \end{bmatrix}.
\end{equation*}
For the linear elasticity problem, the null space comprises the rigid body modes. In 2D, the space of rigid body modes is spanned by two translations
\begin{equation*}
    r_{u, 1} \coloneq \begin{bmatrix}
        1 \\
        0
    \end{bmatrix}, \;
    r_{u, 2} \coloneq \begin{bmatrix}
        0 \\
        1
    \end{bmatrix},
\end{equation*}
and one linearized rotation
\begin{equation*}
    r_{u, 3} \coloneq \begin{bmatrix}
        y \\
        -x
    \end{bmatrix}.
\end{equation*}
In 3D, the null space is spanned by three translations,
\begin{equation*}
    r_{u, 1} \coloneq \begin{bmatrix}
        1 \\
        0 \\
        0
    \end{bmatrix}, \;
    r_{u, 2} \coloneq \begin{bmatrix}
        0 \\
        1 \\
        0
    \end{bmatrix}, \;
    r_{u, 3} \coloneq \begin{bmatrix}
        0 \\
        0 \\
        1
    \end{bmatrix},
\end{equation*}
and three linearized rotations,
\begin{equation*}
    r_{u, 4} \coloneq \begin{bmatrix}
        y \\
        -x \\
        0
    \end{bmatrix}, \;
    r_{u, 5} \coloneq \begin{bmatrix}
        -z \\
        0 \\
        x
    \end{bmatrix}, \;
    r_{u, 6} \coloneq \begin{bmatrix}
        0 \\
        z \\
        -y
    \end{bmatrix}.
\end{equation*}

\subsubsection{The RGDSW coarse space} \label{sec:rgdsw_coarse_space}

The RGDSW coarse space \cite{DOHRMANN2017RGDSW} builds on GDSW by defining fewer coarse basis functions. Let $\mathcal{N}_j$ be a set of nodes shared by exactly the same subdomains, referred to as a nodal equivalence class (NEC). Furthermore, let $S_{\mathcal{N}_j}$ be the index set of subdomains that share the nodes in $\mathcal{N}_j$. A NEC $\mathcal{N}_j$ is an \textit{ancestor} of another NEC $\mathcal{N}_k$ if $S_{\mathcal{N}_k} \subset S_{\mathcal{N}_j}$. Conversely, $\mathcal{N}_k$ is an \textit{offspring} of $\mathcal{N}_j$. A \textit{coarse node} is defined as a NEC that has no ancestors.

Analogously to GDSW, $\Phi_{\Gamma}$ is also derived from the restriction of the null space to the interface components. However, in RGDSW, a smaller coarse space is obtained by grouping the components based on the hierarchy defined by the NECs. First, the restriction of the null space is computed for each coarse node and its offspring. Next, the values of the restriction are multiplied by an interface partition of unity (IPOU) function. This is required by the convergence theory of Schwarz methods \cite{TOSELLI2005DDBOOK}. The interface operator can thus be written as
\begin{equation} \label{eq:rgdsw_int_bf}
    \Phi_{\Gamma}^{\text{RGDSW}} =
    \begin{bmatrix}
        R_{c_1}^{\top} \Phi_{c_1} & \dots & R_{c_n}^{\top} \Phi_{c_n}
    \end{bmatrix} =
    \begin{bmatrix}
        R_{c_1}^{\top} \tilde{S}_{c_1} Z_{c_1} & \dots & R_{c_n}^{\top} \tilde{S}_{c_n} Z_{c_n}
    \end{bmatrix},
\end{equation}
where $Z_{c_i}$ is the restriction of the null space to the coarse node $c_i$ and its offspring, and $\tilde{S}_{c_i} \coloneq \diag{\left( \tilde{R}_{c_i}^{\top} p^{(c_i)} \right)}$ is a diagonal scaling matrix. The vector $p^{(c_i)}$ contains the values of the IPOU function for each node in $c_i$ and its offspring, and $\tilde{R}_{c_i}^{\top}$ maps a vector of nodal quantities to their respective dofs.

An algebraic option of IPOU is given by
\begin{equation*}
    p_j^{(c_i)} =
    \begin{cases}
        1 & \quad \text{if node $j$ is a coarse node}, \\
        \frac{1}{|\mathcal{C}_j|} & \quad \text{otherwise},
    \end{cases}
\end{equation*}
where $\mathcal{C}_j$ is the index set of coarse nodes that are ancestors of the NEC that contains $j$. This is referred to as option 1 in \cite{DOHRMANN2017RGDSW}. A less algebraic alternative based on inverse distance weighting may also be used and is defined as
\begin{equation*}
    p_j^{(c_i)} =
    \begin{cases}
        1 & \quad \text{if node $j$ is a coarse node}, \\
        \frac{1 / d_{c_i}(j)}{\sum\limits_{k \in \mathcal{C}_j} 1 / d_{c_k}(j)} & \quad \text{otherwise},
    \end{cases}
\end{equation*}
where $d_{c_i}(j)$ is a distance measure between node $j$ and the coarse node $c_i$. This is option 2.2 in \cite{DOHRMANN2017RGDSW}.

The product of the null space and the IPOU allows for a more flexible algebraic definition of the coarse space while still ensuring the null space property. This will be explored in the next section for the algebraic multiscale coarse space.

\section{An algebraic multiscale finite element coarse space} \label{sec:ams_coarse_space}

The multiscale finite element method (MsFEM) \cite{HOU1997MSFEM} introduces a set of nodal basis functions designed to capture the multiscale nature of the solution of the diffusion problem in \cref{eq:diffusion_weak_form} by incorporating information about the coefficient into the functions.

Let $\mathcal{T}^H$ be a triangulation of the domain $\Omega$ and $x_i$ an interior node of $\mathcal{T}^H$. For every cell $K \in \mathcal{T}^H$ such that $x_i \in \partial K$, the MsFEM basis functions satisfy
\begin{equation} \label{eq:msfem_bf_def}
    \begin{alignedat}{2}
        \nabla \cdot \left( \alpha (x) \nabla \Phi_i \right) &= 0 \quad \quad && \text{in } K, \\
        \Phi_i &= 0 \quad \quad && \text{on } \partial \Omega_D, \\
        \Phi_i &= \Phi_{i, \Gamma} \quad \quad && \text{on } \Gamma = \partial K \setminus \partial \Omega_D, \\
    \end{alignedat}
\end{equation}
where $\Phi_i$ is the basis function associated with the node $x_i$.

A central part of the local problems is the definition of the boundary data $\Phi_{i, \Gamma}$. A first alternative is to let $\Phi_i$ vary linearly on $\Gamma$. For heterogeneous problems, a more robust choice is to set $\Phi_{i, \Gamma}$ as the solution of a reduced boundary condition (RBC) problem similar to \cref{eq:msfem_bf_def}. This is done by first solving the RBC problem on the edges of $\partial K$ and then using these solutions as the boundary condition for an equivalent problem on the faces. For an edge $\mathcal{E}_{ij} \subset \partial K$ with endpoints $x_i$ and $x_j$, the 1D trace problem is defined as
\begin{equation} \label{eq:msfem_rbc_1d}
    \begin{alignedat}{2}
        \nabla \cdot \left( \alpha (x) \nabla \Phi_{i, \Gamma} \right) &= 0 \quad && \text{on } \mathcal{E}_{ij}, \\
        \Phi_{i, \Gamma}(x_i) &= 1, \\
        \Phi_{i, \Gamma}(x_j) &= 0.
    \end{alignedat}
\end{equation}
The solution of \cref{eq:msfem_rbc_1d} is then employed in an equivalent 2D trace problem on a face $\mathcal{F}$ of $\partial K$:
\begin{equation} \label{eq:msfem_rbc_2d}
    \begin{alignedat}{2}
        \nabla \cdot \left( \alpha (x) \nabla \Phi_{i, \Gamma} \right) &= 0 \quad && \text{on } \mathcal{F}, \\
        \Phi_{i, \Gamma}(x) &= \Phi_{i, \Gamma}|_{\mathcal{E}_{ij}} \quad && \text{for an edge } \mathcal{E}_{ij} \subset \partial \mathcal{F},
    \end{alignedat}
\end{equation}
where $\Phi_{i, \Gamma}|_{\mathcal{E}_{ij}}$ is the value of $\Phi_{i, \Gamma}$ on the edge $\mathcal{E}_{ij}$ as previously computed.

It is important to note that the formulation of the MsFEM basis functions as previously stated requires the definition of subgrids on the interface $\Gamma$ to solve \cref{eq:msfem_rbc_1d} and \cref{eq:msfem_rbc_2d}. In the next subsection, we describe an algebraic procedure to compute the basis functions based on the non-overlapping domain decomposition.

\subsection{Algebraic formulation of the MsFEM coarse basis functions} \label{sec:ams_basis_functions}

A purely algebraic formulation of the MsFEM was first introduced in \cite{LUNATI2009OPMSFV} and further discussed in \cite{WANG2014AMS, ZHOUTCHELEPI2012TAMS} as part of a solver framework applied to the simulation of Darcy flow.

We start by grouping the degrees of freedom (dofs) of the problem into vertices (V), edges (E), faces (F) and interior (I). This is illustrated in \cref{fig:dd_interface}. The system of equations in \cref{eq:les} can thus be rearranged to reflect this structure
\begin{equation} \label{eq:wirebasket_les}
    \begin{bmatrix}
        A_{\text{II}} & A_{\text{IF}} & A_{\text{IE}} & A_{\text{IV}} \\
        A_{\text{FI}} & A_{\text{FF}} & A_{\text{FE}} & A_{\text{FV}} \\
        A_{\text{EI}} & A_{\text{EF}} & A_{\text{EE}} & A_{\text{EV}} \\
        A_{\text{VI}} & A_{\text{VF}} & A_{\text{VE}} & A_{\text{VV}} \\
    \end{bmatrix}
    \begin{bmatrix}
        u_{\text{I}} \\
        u_{\text{F}} \\
        u_{\text{E}} \\
        u_{\text{V}}
    \end{bmatrix} =
    \begin{bmatrix}
        b_{\text{I}} \\
        b_{\text{F}} \\
        b_{\text{E}} \\
        b_{\text{V}}
    \end{bmatrix},
\end{equation}
where, due to the symmetry of $A$, $A_{ij} = A_{ji}^{\top}$, for $i, j \in \{ \text{I, F, E, V} \}$. Given a partition of the domain into non-overlapping subdomains, the decomposition of the interface may still be done algebraically by using the system matrix as an approximation of the mesh connectivity and employing the hierarchy of NECs discussed in \cref{sec:rgdsw_coarse_space} to classify the interface entities. More specifically, a node is a vertex if it belongs to a coarse node NEC, a face if it is in a NEC shared by only two subdomains, or an edge otherwise.

The core idea of the algebraic construction is to represent the RBC problems \cref{eq:msfem_rbc_1d} and \cref{eq:msfem_rbc_2d} using the matrix blocks in \cref{eq:wirebasket_les}. Instead of using a coarse triangulation and auxiliary finer grids to compute the basis functions, we take advantage of the non-overlapping domain decomposition as the coarse partition and then solve the equivalent problems on the original grid.

We begin by noticing that the solution of both \cref{eq:msfem_rbc_1d} and \cref{eq:msfem_rbc_2d} involves only the nodes on interface entities of the same or lower dimension. Moreover, for vertex nodes, the Kronecker delta property may be enforced algebraically by setting $\Phi_{\text{V}} = I_{\text{V}\text{V}}$, where $I_{\text{VV}}$ is an identity matrix of dimension $|\text{V}| \times |\text{V}|$. The resulting system of equations to be solved for $\Phi$ then has the form
\begin{equation*}
    \begin{bmatrix}
        A_{\text{II}} & A_{\text{IF}} & A_{\text{IE}} & A_{\text{IV}} \\
        0 & A_{\text{FF}} & A_{\text{FE}} & A_{\text{FV}} \\
        0 & 0 & A_{\text{EE}} & A_{\text{EV}} \\
        0 & 0 & 0 & I_{\text{VV}} \\
    \end{bmatrix}
    \begin{bmatrix}
        \Phi^{\text{AM}}_{\text{I}} \\
        \Phi^{\text{AM}}_{\text{F}} \\
        \Phi^{\text{AM}}_{\text{E}} \\
        \Phi^{\text{AM}}_{\text{V}}
    \end{bmatrix} =
    \begin{bmatrix}
        0 \\
        0 \\
        0 \\
        I_{\text{VV}}
    \end{bmatrix}.
\end{equation*}
The superscript ``AM'' stands for \textit{algebraic multiscale}.

The diagonal entries of the blocks $A_{\text{FF}}$ and $A_{\text{EE}}$ still take into account the influence of neighboring nodes in interface components of higher dimension and also in the interior. These need to be removed to obtain a discrete system that correctly represents the underlying problem. Let
\begin{equation} \label{eq:ams_diagonal_blocks}
    \begin{aligned}
        \tilde{A}_{\text{FF}} &= A_{\text{FF}} + \diag \left( A_{\text{FI}} \mathbf{1}_{\text{I}} \right), \\
        \tilde{A}_{\text{EE}} &= A_{\text{EE}} + \diag \left( A_{\text{EI}} \mathbf{1}_{\text{I}} \right) + \diag \left( A_{\text{EF}} \mathbf{1}_{\text{F}} \right),
    \end{aligned}
\end{equation}
where $\mathbf{1}_{\text{I}}$ and $\mathbf{1}_{\text{F}}$ are vectors of dimension $|\text{I}|$ and $|\text{F}|$, respectively, whose entries are all ones. The matrices $\tilde{A}_{\text{FF}}$ and $\tilde{A}_{\text{EE}}$ are block diagonal and correspond to the blocks $A_{\text{FF}}$ and $A_{\text{EE}}$, respectively, with the influence of the lower triangular blocks removed from the main diagonal. This procedure guarantees that the energy-minimality of constant functions is maintained, thus ensuring that a partition of unity is obtained, an important result in the context of Schwarz methods \cite{TOSELLI2005DDBOOK}.

The final form of the algebraic problem that defines the AM coarse basis functions is then given by
\begin{equation} \label{eq:les_rbc_ams}
    \begin{bmatrix}
        A_{\text{II}} & A_{\text{IF}} & A_{\text{IE}} & A_{\text{IV}} \\
        0 & \tilde{A}_{\text{FF}} & A_{\text{FE}} & A_{\text{FV}} \\
        0 & 0 & \tilde{A}_{\text{EE}} & A_{\text{EV}} \\
        0 & 0 & 0 & I_{\text{VV}} \\
    \end{bmatrix}
    \begin{bmatrix}
        \Phi^{\text{AM}}_{\text{I}} \\
        \Phi^{\text{AM}}_{\text{F}} \\
        \Phi^{\text{AM}}_{\text{E}} \\
        \Phi^{\text{AM}}_{\text{V}}
    \end{bmatrix} =
    \begin{bmatrix}
        0 \\
        0 \\
        0 \\
        I_{\text{VV}}
    \end{bmatrix}.
\end{equation}
The system in \cref{eq:les_rbc_ams} can be solved by block backward substitution leading to
\begin{equation} \label{eq:phi_ams}
    \Phi^{\text{AM}} =
    \begin{bmatrix}
        \Phi^{\text{AM}}_{\text{I}} \\
        \Phi^{\text{AM}}_{\text{F}} \\
        \Phi^{\text{AM}}_{\text{E}} \\
        \Phi^{\text{AM}}_{\text{V}}
    \end{bmatrix} =
    \begin{bmatrix}
        -A_{\text{II}}^{-1} \left( A_{\text{IV}} \Phi^{\text{AM}}_{\text{V}} + A_{\text{IE}} \Phi^{\text{AM}}_{\text{E}} + A_{\text{IF}} \Phi^{\text{AM}}_{\text{F}} \right) \\
        -\tilde{A}_{\text{FF}}^{-1} \left( A_{\text{FV}} \Phi^{\text{AM}}_{\text{V}} + A_{\text{FE}} \Phi^{\text{AM}}_{\text{E}} \right) \\
        -\tilde{A}_{\text{EE}}^{-1} A_{\text{EV}} \Phi^{\text{AM}}_{\text{V}} \\
        I_{\text{VV}}
    \end{bmatrix}.
\end{equation}
Following the same approach, we can derive the 2D version of the prolongation operator as
\begin{equation*}
    \Phi^{\text{AM}} =
    \begin{bmatrix}
        \Phi^{\text{AM}}_{\text{I}} \\
        \Phi^{\text{AM}}_{\text{E}} \\
        \Phi^{\text{AM}}_{\text{V}}
    \end{bmatrix} =
    \begin{bmatrix}
        -A_{\text{II}}^{-1} \left( A_{\text{IV}} \Phi^{\text{AM}}_{\text{V}} + A_{\text{IE}} \Phi^{\text{AM}}_{\text{E}} \right) \\
        -\tilde{A}_{\text{EE}}^{-1} A_{\text{EV}} \Phi^{\text{AM}}_{\text{V}} \\
        I_{\text{VV}}
    \end{bmatrix},
\end{equation*}
where $\tilde{A}_{\text{EE}}$ is obtained as in \cref{eq:ams_diagonal_blocks} ignoring the term related to face nodes.

Finally, we show that the algebraic multiscale coarse basis functions form a partition of unity.
\begin{lemma} \label{lemma:ams_pou}
    The coarse basis functions defined in \cref{eq:phi_ams} form a partition of unity, i.e., $\Phi^{\text{AM}} \mathbf{1}_{\text{V}} = \mathbf{1}$.
\end{lemma}
\begin{proof}
    We proceed by inspecting the contribution of each group of interface entities. For vertex nodes, it follows from the Kronecker delta property of the multiscale functions that $\Phi^{\text{AM}}_{\text{V}} \mathbf{1}_{\text{V}} = I_{\text{VV}} \mathbf{1}_{\text{V}} = \mathbf{1}_{\text{V}}$.
    
    For edge nodes, the partition of unity property is satisfied if
    \begin{equation*}
        \Phi^{\text{AM}}_{\text{E}} \mathbf{1}_{\text{V}} = -\tilde{A}_{\text{EE}}^{-1} A_{\text{EV}} \mathbf{1}_{\text{V}} = \mathbf{1}_{\text{E}},
    \end{equation*}
    or equivalently
    \begin{equation*}
        \tilde{A}_{\text{EE}} \mathbf{1}_{\text{E}} + A_{\text{EV}} \mathbf{1}_{\text{V}} = \mathbf{0}.
    \end{equation*}
    Replacing the definition of $\tilde{A}_{\text{EE}}$ from \cref{eq:ams_diagonal_blocks}, we have
    \begin{align*}
        \tilde{A}_{\text{EE}} \mathbf{1}_{\text{E}} + A_{\text{EV}} \mathbf{1}_{\text{V}}
        &= A_{\text{EE}} \mathbf{1}_{\text{E}} + \diag \left( A_{\text{EI}} \mathbf{1}_{\text{I}} \right) \mathbf{1}_{\text{E}} + \diag \left( A_{\text{EF}} \mathbf{1}_{\text{F}} \right) \mathbf{1}_{\text{E}} + A_{\text{EV}} \mathbf{1}_{\text{V}} \\
        &= A_{\text{EE}} \mathbf{1}_{\text{E}} + A_{\text{EI}} \mathbf{1}_{\text{I}} + A_{\text{EF}} \mathbf{1}_{\text{F}} + A_{\text{EV}} \mathbf{1}_{\text{V}},
    \end{align*}
    where the $\diag$ operator can be removed as the multiplication by the vector of ones retrieves the original vector that was expanded into a diagonal matrix. The term above is the sum of the entries in each row of the system matrix related to edge nodes. Since the problem is discretized using piecewise linear finite elements, the sum of the entries of a row corresponding to a node not on the Dirichlet boundary is zero. Therefore, we have a partition of unity for edge nodes.
    
    We may proceed similarly for face nodes:
    \begin{align*}
        \tilde{A}_{\text{FF}} \mathbf{1}_{\text{F}} + A_{\text{FV}} \mathbf{1}_{\text{V}} + A_{\text{FE}} \Phi^{\text{AM}}_{\text{E}} \mathbf{1}_{\text{V}}
        &= A_{\text{FF}} \mathbf{1}_{\text{F}} + \diag \left( A_{\text{FI}} \mathbf{1}_{\text{I}} \right) \mathbf{1}_{\text{F}} + A_{\text{FV}} \mathbf{1}_{\text{V}} + A_{\text{FE}} \mathbf{1}_{\text{E}} \\
        &= A_{\text{FF}} \mathbf{1}_{\text{F}} + A_{\text{FI}} \mathbf{1}_{\text{I}} + A_{\text{FV}} \mathbf{1}_{\text{V}} + A_{\text{FE}} \mathbf{1}_{\text{E}} \\
        &= \mathbf{0}.
    \end{align*}
    Thus, the partition of unity property is satisfied for face nodes as well.

    Repeating the same steps for the interior nodes, we have
    \begin{align*}
        A_{\text{II}} \mathbf{1}_{\text{I}} + \left( A_{\text{IV}} \Phi^{\text{AM}}_{\text{V}} + A_{\text{IE}} \Phi^{\text{AM}}_{\text{E}} + A_{\text{IF}} \Phi^{\text{AM}}_{\text{F}} \right) \mathbf{1}_{\text{V}}
        &= A_{\text{II}} \mathbf{1}_{\text{I}} + A_{\text{IV}} \mathbf{1}_{\text{V}} + A_{\text{IE}} \mathbf{1}_{\text{E}} + A_{\text{IF}} \mathbf{1}_{\text{F}} \\
        &= \mathbf{0}.
    \end{align*}
    Hence $\Phi^{\text{AM}}_{\text{I}} \mathbf{1}_{\text{V}} = \mathbf{1}_{\text{I}}$.
\end{proof}

\begin{remark} \label{re:ams_pou}
    The assignment of the Dirichlet boundary conditions in the system matrix affects the partition of unity property of the coarse basis functions. If the Dirichlet conditions are assigned by setting the rows and columns of the system matrix corresponding to the boundary nodes to zero and placing ones on the diagonal, then, for the interface entities with nodes neighboring the Dirichlet boundary, the partition of unity property is not satisfied. This is because the sum of the entries in the rows corresponding to these nodes is no longer zero. Nevertheless, the partition of unity may be enforced by normalizing the edge and face values during the computation of the prolongation operator.
\end{remark}

\subsection{Computational complexity analysis} \label{sec:comp_analysis}

Before introducing the theoretical results on the AM coarse space, we discuss the asymptotic time complexity to compute the coarse basis functions defined in \cref{eq:phi_ams} and set up the two-level Schwarz preconditioner in \cref{eq:two_level_oas_preconditioner}.

Let $n_{\text{V}}$, $n_{\text{E}}$, $n_{\text{F}}$, and $n_{\text{I}}$ be the number of vertex, edge, face, and interior nodes, respectively. In addition, let $N$ be the number of non-overlapping subdomains, $n_{\Omega_i}$ the maximum number of nodes in an overlapping subdomain, and $nnz(B)$ the number of non-zero entries in a sparse matrix $B$.

We split the computations according to the node classes. For vertex nodes, only an identity matrix is initialized in $O(n_{\text{V}})$ operations. For edge nodes, the computations can be divided into two main steps:
\begin{enumerate}
    \item Set $\tilde{A}_{\text{EE}} = A_{\text{EE}} + \diag \left( A_{\text{EI}} \mathbf{1}_{\text{I}} \right) + \diag \left( A_{\text{EF}} \mathbf{1}_{\text{F}} \right)$; and
    \item Solve $\tilde{A}_{\text{EE}} \Phi_{\text{E}} = A_{\text{EV}}$.
\end{enumerate}
The first step may be implemented with one sparse matrix addition, two sparse matrix-vector products, and one vector addition. This requires $O(nnz(A_{\text{EI}}) + nnz(A_{\text{EF}}) + nnz(A_{\text{EE}}) + n_{\text{E}})$ operations. The second step involves the solution of a linear system with an SPD matrix for $n_{\text{V}}$ right-hand side vectors. It can be implemented with $O(n_{\text{V}} n_{\text{E}}^2)$ operations.

Proceeding similarly for face nodes, we need to perform the following steps:
\begin{enumerate}
    \item Set $\tilde{A}_{\text{FF}} = A_{\text{FF}} + \diag \left( A_{\text{FI}} \mathbf{1}_{\text{I}} \right)$;
    \item Compute $A_{\text{FE}} \Phi_{\text{E}} + A_{\text{FV}} \Phi_{\text{V}}$; and
    \item Solve $\tilde{A}_{\text{FF}} \Phi_{\text{F}} = A_{\text{FE}} \Phi_{\text{E}} + A_{\text{FV}} \Phi_{\text{V}}$.
\end{enumerate}
Step 1 involves one sparse matrix addition and one sparse matrix-vector product, thus requiring $O(nnz(A_{\text{FI}}) + nnz(A_{\text{FF}}) + n_{\text{F}})$ operations. Step 2 can be implemented with one sparse matrix-sparse matrix product and one sparse matrix addition. The term $A_{\text{FV}} \Phi_{\text{V}}$ does not need to be explicitly evaluated since $\Phi_{\text{V}} = I_{\text{VV}}$. Therefore, $O(nnz(A_{\text{FE}}) nnz(\Phi_{\text{E}}) + nnz(A_{\text{FE}} \Phi_{\text{E}}) + nnz(A_{\text{FV}}))$ operations are performed. In the last step, the solution of the linear system can be done in $O(n_{\text{V}} n_{\text{F}}^2)$ operations.

For interior nodes, the main stages are:
\begin{enumerate}
    \item Compute $A_{\text{IF}} \Phi_{\text{F}}$;
    \item Compute $A_{\text{IE}} \Phi_{\text{E}}$;
    \item Perform the addition $A_{\text{IF}} \Phi_{\text{F}} + A_{\text{IE}} \Phi_{\text{E}} + A_{\text{IV}} \Phi_{\text{V}}$; and
    \item Solve $A_{\text{II}} \Phi_{\text{I}} = A_{\text{IF}} \Phi_{\text{F}} + A_{\text{IE}} \Phi_{\text{E}} + A_{\text{IV}} \Phi_{\text{V}}$.
\end{enumerate}
Steps 1 and 2 demand $O(nnz(A_{\text{IF}}) nnz(\Phi_{\text{F}}))$ and $O(nnz(A_{\text{IE}}) nnz(\Phi_{\text{E}}))$ operations. Next, step 3 involves $O(nnz(A_{\text{IF}} \Phi_{\text{F}}) + nnz(A_{\text{IE}} \Phi_{\text{E}}) + nnz(A_{\text{IV}}))$ operations.

Finally, looking into the Schwarz preconditioner in \cref{eq:two_level_oas_preconditioner}, its setup involves the factorization of the $N$ local subproblems and of the coarse matrix, which require $O(Nn_{\Omega_i}^2)$ and $O(n_{\text{V}}^2)$ operations, respectively. Moreover, for the application of the preconditioner, $O(Nn_{\Omega_i}^2 + n_{\text{V}}^2)$ operations are performed corresponding to the solution of the local subproblems and the coarse problem.

\subsection{Convergence analysis of the scalar diffusion case} \label{sec:conv_analysis_ams}

In this subsection, we review the condition number bound for multiscale coarse spaces discussed in \cite{GRAHAM2007MSPDES} and its application to our algebraic multiscale coarse space. We closely follow the proof presented there for the 2D scalar diffusion case.

The upper bound on the condition number of the two-level overlapping Schwarz operator can be expressed as
\begin{equation} \label{eq:cond_num_bound_msfem}
    \kappa \left( M_{\text{OAS-2}}^{-1} A \right) \lesssim \pi \left( \alpha \right) \gamma \left( 1 \right) \max_{1 \leq i \leq N} \left( 1 + \frac{H_i}{\delta_i} \right) + \gamma \left( \alpha \right),
\end{equation}
where $H_i$ and $\delta_i$ are the diameter and the overlap size related to subdomain $\Omega_i$, and $\pi \left( \alpha \right)$ and $\gamma (\alpha)$ are the partitioning and the coarse space robustness indicators, respectively. These functions evaluate how well-chosen the overlapping domain decomposition and the coarse basis functions are with respect to some coefficient distribution $\alpha$. Moreover, the relation $a \lesssim b$ means that $a \leq Cb$ for some constant $C > 0$ independent of the mesh size, the subdomain parameters and the coefficient function.

The partitioning robustness indicator is defined by
\begin{equation} \label{eq:partition_robust_ind}
    \pi (\alpha) = \inf_{\{ \chi_i \} \in \Pi(\{ \Omega_i' \})}{\pi \left( \alpha, \{ \chi_i \} \right)},
\end{equation}
where $\Pi(\{ \Omega_i' \})$ denotes the set of all partitions of unity $\{ \chi_i \}$ subordinate to the overlapping domain decomposition, and
\begin{equation} \label{eq:partition_robust_ind_pou}
    \pi \left( \alpha, \{ \chi_i \} \right) = \max_{i = 1}^N{\delta_i^2 || \alpha | \nabla \chi_i |^2 ||_{L_{\infty}(\Omega)}}
\end{equation}
is the robustness indicator for a specific partition of unity. Since \cref{eq:partition_robust_ind} and \cref{eq:partition_robust_ind_pou} depend only on the overlapping subdomains, the partitioning robustness indicator is not affected by changes in the coarse space.

The coarse space robustness indicator, in turn, is given by
\begin{equation*}
    \gamma (\alpha) \coloneq \max_{p \in \mathcal{I}^H}{\left\{ 1 + H_p^{2 - d} |\Phi_p|_{H^1(\Omega), \alpha}^2 \right\}},
\end{equation*}
where $\mathcal{I}^H$ is the index set of coarse nodes in $\mathcal{T}^H$, $d$ is the dimension of the problem (e.g., $d = 2$ in 2D), $|u|_{H^1(\Omega), \alpha}^2 \coloneq \int_{\Omega} \alpha |\nabla u|^2 \;\text{d}x$, and $H_p \coloneq \diam{\omega_p}$ with $\omega_p$ being the set of all coarse cells $K \in \mathcal{T}^H$ that have the coarse node $x_p$ on their closure.

\begin{figure}[t]
  \centering
  \includegraphics[width=0.75\textwidth]{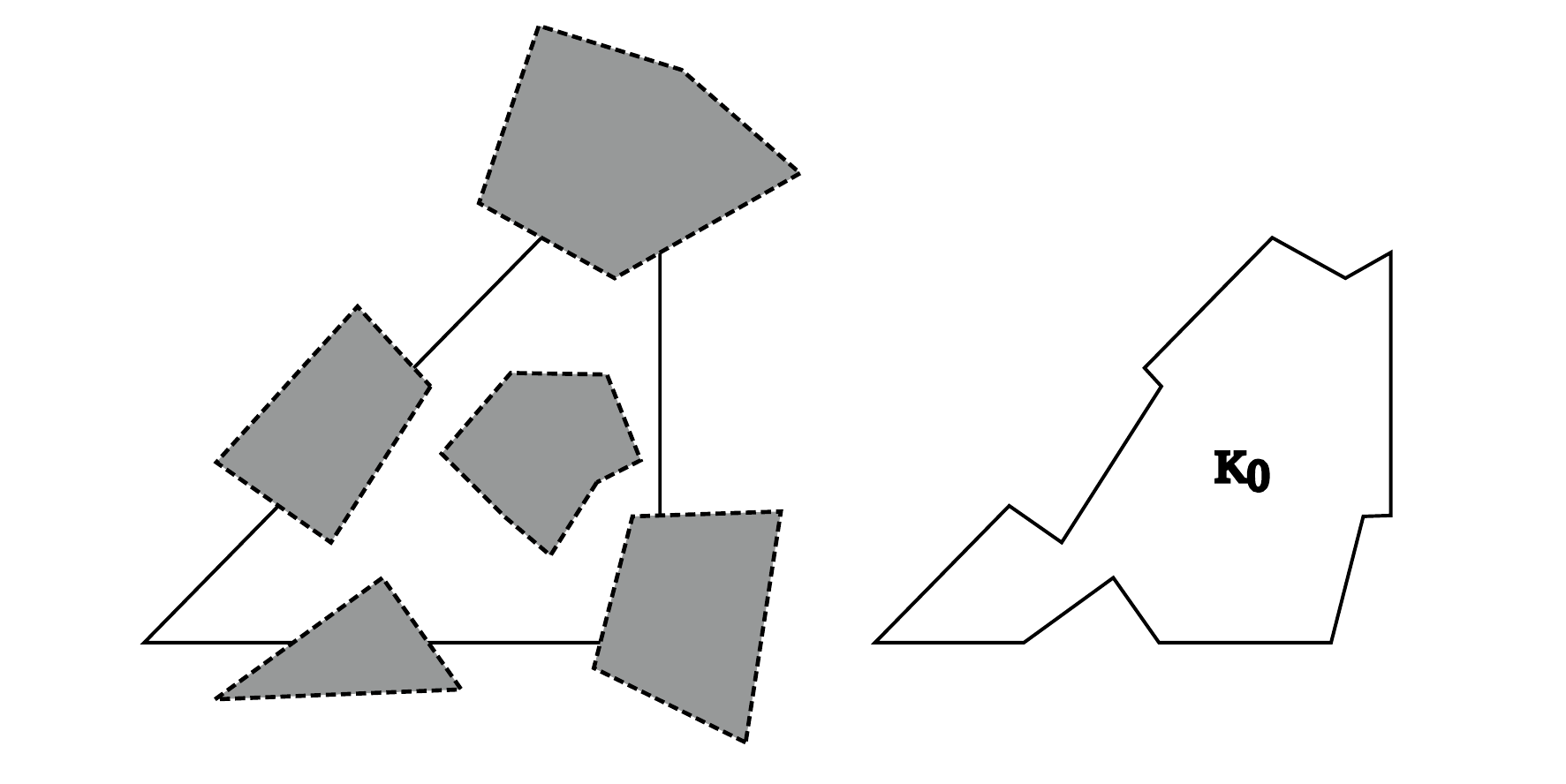}
  \caption{Illustration of the components $K(\eta)$ on a coarse cell $K$ (left) and the corresponding region $K_0$ excluding the boundary components $K^B(\eta)$ (right). Adapted from \cite{GRAHAM2007MSPDES}.}
  \label{fig:coarse_cell_coeff_islands}
\end{figure}

Before discussing an upper bound on $\gamma(\alpha)$ for the MsFEM coarse basis functions, we first introduce the notation used in \cite{GRAHAM2007MSPDES}. Let $K \in \mathcal{T}^H$ be a cell on a coarse triangulation as in \cref{eq:msfem_bf_def}. We define the set
\begin{equation*}
    K(\eta) \coloneq \left\{ x \in K \, | \, \alpha(x) > \eta \right\}
\end{equation*}
for some arbitrary constant $\eta \geq 1$. The set $K(\eta)$ may be decomposed as
\begin{equation*}
    K(\eta) = K^I(\eta) \cup K^B(\eta),
\end{equation*}
where $K^B(\eta)$ contains the components of $K(\eta)$ whose closure touches $\partial K$, and $K^I(\eta)$ contains the interior components. We also define the distance measure
\begin{equation*}
    \tilde{\epsilon} \left( \eta, K \right) \coloneq \dist{\left( K^I \left( \eta \right), \partial K \cup K^B \left( \eta \right) \right)}.
\end{equation*}
In addition, some assumptions on $K(\eta)$ and $\tilde{\epsilon} \left( \eta, K \right)$ are required:
\begin{enumerate}
    \item $K^B(\eta)$ and $K^I(\eta)$ should be well separated and $K_0 \coloneq K \setminus K^B(\eta)$ should contain a sufficiently large part of $K$ such that $\tilde{\epsilon} \left( \eta, K \right) > \frac{3h}{2}$ and $|K_0| \gtrsim H_K^2$.
    \item We may write $K^B(\eta) = \bigcup_{l = 1}^L K^B_l(\eta)$, where each $K^B_l(\eta)$ is simply connected and pairwise disjoint, and $\alpha$ is constant on the closure of $K^B_l(\eta)$, i.e., $\alpha(x) = \alpha_l,\, x \in \overline{K^B_l(\eta)}$.
    \item The sets $\Gamma_l^B(\eta) \coloneq \overline{K^B_l(\eta)} \cap \partial K$ are also simply connected and $\Gamma^B(\eta) \coloneq \bigcup_{l = 1}^L \Gamma^B_l(\eta)$ does not cover much of any edge $e$ of $\partial K$, i.e., $| e \setminus \Gamma^B(\eta) | \gtrsim \tilde{\epsilon} \left( \eta, K \right)$.
    \item $K_l^B(\eta)$ has a polygonal boundary such that each side has length $\gtrsim H_K$.
    \item The coefficient $\alpha$ is continuous as one crosses the boundary of an element $K$.
\end{enumerate}
An example of a valid decomposition $K(\eta)$ is illustrated in \cref{fig:coarse_cell_coeff_islands}.

If the assumptions above hold true for each $K$, then $\gamma(\alpha)$ can be bounded from above by
\begin{equation} \label{eq:coarse_space_ind_msfem}
    \gamma \left( \alpha \right) \lesssim \eta \max_{K \in \mathcal{T}^H} \left\{ \left( \frac{H_K^2}{\tilde{\epsilon} \left( \eta, K \right)} \right)^2 \left( 1 + \log{\frac{\eta H_K}{\tilde{\epsilon} \left( \eta, K \right)}} \right) \right\} + \eta^2 \max_{K \in \mathcal{T}^H} \left\{ \frac{H_K}{\tilde{\epsilon} \left( \eta, K \right)} \right\}.
\end{equation}
We note that \cref{eq:coarse_space_ind_msfem} remains bounded even if $\alpha_{max} \to \infty$.

For the algebraic variant of the MsFEM coarse basis functions, we need an additional assumption on $K$ if it is formed by the union of fine-scale elements:
\begin{enumerate}
    \setcounter{enumi}{5}
    \item $\partial K$ is \textit{algebraically connected}, i.e., given the graph induced by the system matrix $A$, for each pair of nodes $x_m, x_n \in \partial K$, there exists a path $\{ x_{i_1}, x_{i_2}, \ldots, x_{i_k} \}$ such that $x_{i_1} = x_m$, $x_{i_k} = x_n$, and for two consecutive nodes in the path, $A_{i_{j} i_{j+1}} \neq 0$.
\end{enumerate}
This is a generalization to the algebraic setting of the requirement that $\partial K$ be simply connected. Moreover, it is a necessary condition to ensure that the submatrices $\tilde{A}_{\text{EE}}$ and $\tilde{A}_{\text{FF}}$ are invertible.

To verify that the algebraic multiscale coarse space satisfies \cref{eq:cond_num_bound_msfem}, it suffices to show that the following assumptions on the coarse basis functions $\Phi_j$ are satisfied for each coarse node $x_j$:
{
    \renewcommand{\labelenumi}{(\textbf{A\arabic{enumi}})}
    \begin{enumerate}
        \item $\Phi_j \in S^h(\Omega)$, where $S^h(\Omega)$ is the space of piecewise linear functions with respect to the fine triangulation of the domain $\Omega$;
        \item $\Phi_j(x_k) = \delta_{jk}$ for every coarse node $x_k$;
        \item $\supp{\Phi_j} \subset \overline{\omega}_j$, where $\omega_j$ is the union of all subdomains that contain the coarse node $j$;
        \item $\sum_{j} \Phi_j (x) = 1, \, x \in \overline{\Omega}$; and
        \item $|| \Phi_j ||_{L_{\infty}(\Omega)} \lesssim 1$.
    \end{enumerate}
}
Additionally, to ensure that the bound on the coarse space robustness indicator is also valid, we need to impose the following condition on the boundary data $\Phi_{j, \Gamma}$ in \cref{eq:msfem_bf_def}:
{
    \renewcommand{\labelenumi}{(\textbf{A\arabic{enumi}})}
    \begin{enumerate}
        \setcounter{enumi}{5}
        \item $\Phi_{j, \Gamma}(x) = 0$ on any edge of $\partial K$ that does not contain the coarse node $j$.
    \end{enumerate}
}

We can now verify that the assumptions on the coarse basis functions are satisfied by the algebraic multiscale coarse space.

\begin{theorem}
    The condition number of the algebraic multiscale two-level overlapping Schwarz operator is bounded by \cref{eq:cond_num_bound_msfem} with the coarse space robustness indicator given by \cref{eq:coarse_space_ind_msfem}.
\end{theorem}
\begin{proof}
    We begin by noticing that assumptions (A1) and (A2) are satisfied by definition of the prolongation operator, as $\Phi^{\text{AM}}_{\text{V}} = I_{\text{VV}}$ and the RBC problems are solved on the space of piecewise linear functions. Additionally, from \cref{lemma:ams_pou}, we have that the partition of unity property is met, thus ensuring that assumption (A4) is true.

    To verify assumption (A3), we recall that the submatrices $\tilde{A}_{\text{EE}}$, $\tilde{A}_{\text{FF}}$, and $A_{\text{II}}$ are block diagonal. For a given coarse basis function $\Phi_j$, we may group the diagonal blocks related to nodes inside and outside $\omega_j$. On the edges, we have
    \begin{align*}
        \tilde{A}_{\text{EE}} \Phi_{\text{E}, j}^{\text{AM}} &= -A_{\text{EV}, j} \\
        \implies
        \begin{bmatrix}
            \tilde{A}_{\text{EE}}^{\omega_j} & 0 \\
            0 & \tilde{A}_{\text{EE}}^{\Omega \setminus \omega_j}
        \end{bmatrix}
        \begin{bmatrix}
            \Phi_{\text{E}, j}^{\text{AM}, \omega_j} \\
            \Phi_{\text{E}, j}^{\text{AM}, \Omega \setminus \omega_j}
        \end{bmatrix} &=
        \begin{bmatrix}
            -A_{\text{EV}, j}^{\omega_j} \\
            -A_{\text{EV}, j}^{\Omega \setminus \omega_j}
        \end{bmatrix},
    \end{align*}
    where $A_{\text{EV}, j}$ is the $j$-th column of $A_{\text{EV}}$. For nodes in $\Omega \setminus \omega_j$,
    \begin{equation*}
        \Phi_{\text{E}, j}^{\text{AM}, \Omega \setminus \omega_j} = -\left( \tilde{A}_{\text{EE}}^{\Omega \setminus \omega_j} \right)^{-1} A_{\text{EV}, j}^{\Omega \setminus \omega_j}.
    \end{equation*}
    Since no node outside $\omega_j$ is connected to the coarse node $x_j$, $A_{\text{EV}, j}^{\Omega \setminus \omega_j} = 0$ and, as a consequence, $\Phi_{\text{E}, j}^{\text{AM}, \Omega \setminus \omega_j} = 0$. Therefore, $\Phi_{\text{E}, j}^{\text{AM}}$ has nonzero values only in $\omega_j$. We may proceed similarly for the remaining components $\Phi_{\text{F}, j}^{\text{AM}}$ and $\Phi_{\text{I}, j}^{\text{AM}}$ to obtain the same result and verify that $\supp{\Phi_j} \subset \overline{\omega}_j$. Moreover, as a consequence of the Kronecker delta property of $\Phi_j$ and the fact that the right-hand side of the local problem for nodes on $\partial \omega_j$ is zero, we have that $\Phi_j(x) = 0$ on $\partial \omega_j$ and thus assumption (A6) is also satisfied.

    Finally, for assumption (A5), we invoke the fact that, for any piecewise constant coefficient function, if no element $K \in \mathcal{T}^H$ has an obtuse angle, then the energy-minimizing extension of the boundary data $\Phi_{j, \Gamma}$ respects the discrete maximum principle; cf. \cite[Remark~4.2]{GRAHAM2007MSPDES}. From the boundary conditions of the RBC problem, the maximum value of $\Phi_j$ is attained at the coarse node $x_j$ and is equal to 1. Therefore, $|| \Phi_j ||_{L_{\infty}(\Omega)} = 1$. Additionally, as previously shown, the minimum value is attained on $\partial \omega_j$ and is equal to 0. Hence, $0 \leq \Phi_{j} \leq 1$.
\end{proof}

\subsection{Handling algebraically unconnected interface entities} \label{sec:slab_fix}

As discussed in \cref{sec:conv_analysis_ams}, the application of the algebraic multiscale procedure requires the interface of the non-overlapping domain decomposition to be algebraically connected. This condition is necessary for the blocks $\tilde{A}_{\text{EE}}$ and $\tilde{A}_{\text{FF}}$ obtained after the diagonal modification in \cref{eq:ams_diagonal_blocks} to be non-singular. However, a partition provided as input does not necessarily satisfy this requirement. In this subsection, we describe a modification that allows the AM method to be employed in such cases.

Consider the edge entity in \cref{fig:unconnected_interface} and a diffusion problem in a homogeneous medium. Since the entries of the system matrix coupling the edge nodes to each other and to the adjacent vertices are all zero, the corresponding submatrices in the AM method are
\begin{equation*}
    A_{ee} =
    \begin{bmatrix}
        4 & 0 & 0 \\
        0 & 4 & 0 \\
        0 & 0 & 4
    \end{bmatrix},\,
    A_{ev} =
    \begin{bmatrix}
        0 & 0 \\
        0 & 0 \\
        0 & 0
    \end{bmatrix},
\end{equation*}
where the subscripts $e$ and $v$ denote, respectively, the nodes on the edge in \cref{fig:unconnected_interface} and the vertex nodes at the endpoints of that edge. Since each row of $A$ sums to zero, it follows that $\tilde{A}_{ee} = A_{ee} + \diag \left( A_{e\text{I}} \mathbf{1}_{\text{I}} \right) = \mathbf{0}$. Moreover, since $A_{ev} = \mathbf{0}$, even if $\tilde{A}_{ee}$ were non-singular, the coarse basis functions would vanish on $e$.

\begin{figure}[t]
    \centering
    \begin{subfigure}[b]{0.3\textwidth}
        \includegraphics[width=\textwidth]{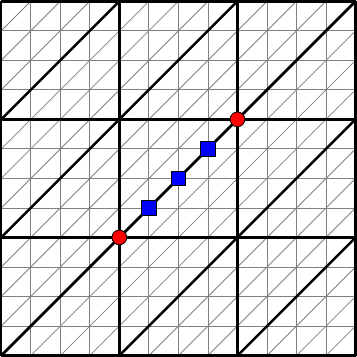}
        \caption{}
        \label{fig:unconnected_interface}
    \end{subfigure}
    \hfil
    \begin{subfigure}[b]{0.3\textwidth}
        \includegraphics[width=\textwidth]{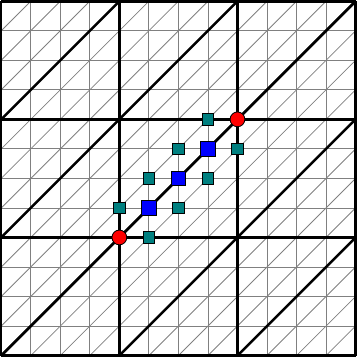}
        \caption{}
        \label{fig:slab_interface}
    \end{subfigure}
    \caption{Illustration of a non-overlapping domain decomposition highlighting an algebraically unconnected edge (a) and the slab nodes used to compute a valid interface partition of unity (b). The vertex, edge, and additional interior nodes are illustrated as red circles, blue squares, and dark green squares, respectively.}
    \label{fig:slab_fix}
\end{figure}

To address this issue, we extend the interface entity by one or more layers of nodes, thus including a slab of neighboring nodes, as illustrated in \cref{fig:slab_interface}. Layers are added until the slab is algebraically connected, which ensures that the submatrices associated with the slab are non-singular. This approach is related to the coarse space described in \cite[subsection~4.2]{HEINLEIN2018MS} based on the ACMS method.

Given the slab region, we solve an algebraic RBC problem analogous to \cref{eq:phi_ams}. For an edge entity $e \subset \Gamma$, we define the corresponding slab region $e'$ and compute the coarse basis functions $\Phi^{\text{AM}}_{e'}$ on $e'$ as
\begin{equation} \label{eq:ams_slab_bf_edge}
    \Phi^{\text{AM}}_{e'} = -\tilde{A}_{e' e'}^{-1} A_{e' \text{V}_e} \Phi^{\text{AM}}_{\text{V}_e}, \quad \text{with } \tilde{A}_{e' e'} = A_{e' e'} +  \diag \left( A_{e' e^{\ast}} \mathbf{1}_{e^{\ast}} \right),
\end{equation}
where $e^{\ast} \coloneq \omega_{e'} \setminus \left( e' \cup \text{V}_e \right)$ with $\text{V}_e$ and $\omega_{e'}$ being the set of vertex nodes that are ancestors of $e$ and the set of nodes in the subdomains that share $e'$, respectively. Note that $e'$ may contain nodes belonging to other interface entities. Since the diagonal modification is performed over $e^{\ast}$ and the extension is computed only from $\text{V}_e$, no spurious coupling is introduced between $e$ and other vertices. Similarly, for a face entity $f \subset \Gamma$, we define the corresponding slab region $f'$ and compute the coarse basis functions $\Phi^{\text{AM}}_{f'}$ on $f'$ as
\begin{equation*}
    \Phi^{\text{AM}}_{f'} = -\tilde{A}_{f' f'}^{-1} \left( A_{f' \text{V}_f} \Phi^{\text{AM}}_{\text{V}_f} + A_{f'\text{E}_f} \Phi^{\text{AM}}_{\text{E}_f} \right), \quad \text{with } \tilde{A}_{f' f'} = A_{f' f'} + \diag \left( A_{f' f^{\ast}} \mathbf{1}_{f^{\ast}} \right),
\end{equation*}
where $f^{\ast} \coloneq \omega_{f'} \setminus \left( f' \cup \text{V}_f \cup \text{E}_f \right)$ with $\text{V}_f$ and $\text{E}_f$ being the sets of vertex and edge nodes that are ancestors of $f$, respectively, and $\omega_{f'}$ the set of nodes in the subdomains that share $f'$. Once the coarse basis functions on the slab are computed, only the values on the initial non-extended entity are kept. Therefore, the original support of the interface values is restored and the extension to the interior can be computed as defined in \cref{eq:phi_ams}. An example of a basis function on a slab edge $e'$ and its restriction to the original edge $e$ is shown in \cref{fig:slab_ipou}.

\begin{figure}[t]
    \centering
    \begin{subfigure}[b]{0.495\textwidth}
        \includegraphics[width=\textwidth]{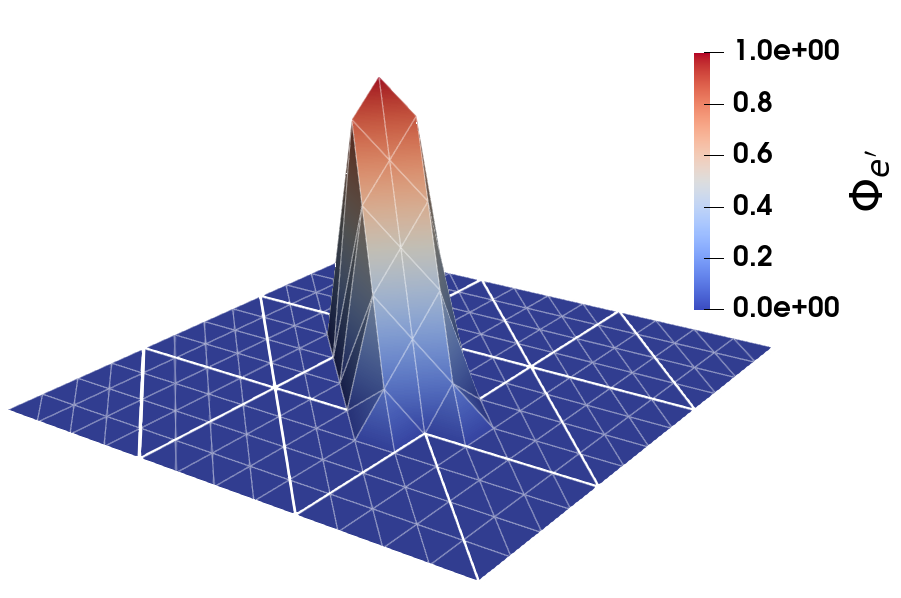}
        \caption{}
        \label{fig:slab_ipou_all}
    \end{subfigure}
    \hfil
    \begin{subfigure}[b]{0.495\textwidth}
        \includegraphics[width=\textwidth]{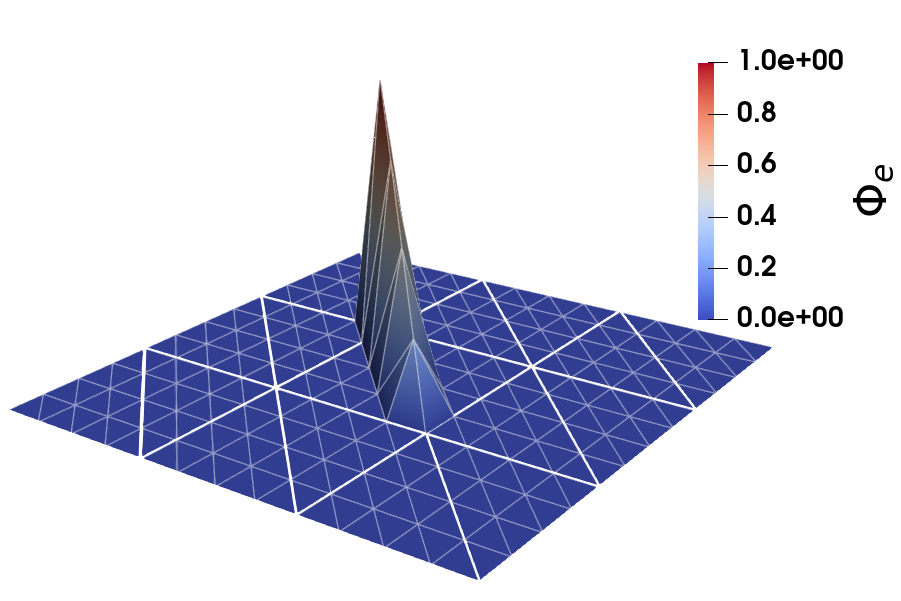}
        \caption{}
        \label{fig:slab_ipou_edge}
    \end{subfigure}
    \caption{Visualization of the edge component of an AM coarse basis function computed on a slab (left) and its restriction to the original entity (right).}
    \label{fig:slab_ipou}
\end{figure}

An additional benefit of solving the RBC problem on the slab entity is that the interface trace values incorporate information from beyond the initial edge or face. Therefore, possible changes in the coefficient that would otherwise not be represented by the original AM functions are captured by the slab.

For completeness, we show that the partition of unity property is preserved on the slab region and, consequently, on the original interface entity.

\begin{lemma} \label{lemma:ams_slab_pou}
    Let $\Gamma_i \subset \Gamma$ be an edge or a face entity, and let $\Gamma_i'$ be the corresponding slab region. The coarse basis functions $\Phi^{\text{AM}}_{\Gamma_i'}$ computed on $\Gamma_i'$ form a partition of unity, i.e., $\Phi^{\text{AM}}_{\Gamma_i'} \mathbf{1}_{\text{V}} = \mathbf{1}_{\Gamma_i'}$. In addition, the restriction $\Phi^{\text{AM}}_{\Gamma_i}$ of $\Phi^{\text{AM}}_{\Gamma_i'}$ to $\Gamma_i$ forms a partition of unity, i.e., $\Phi^{\text{AM}}_{\Gamma_i} \mathbf{1}_{\text{V}} = \mathbf{1}_{\Gamma_i}$.
\end{lemma}
\begin{proof}
    As in the proof of \cref{lemma:ams_pou}, it suffices to verify that the row sums of the submatrices coupling the slab region to the remaining nodes vanish. For a slab edge $e'$, the definition of $\tilde{A}_{e' e'}$ in \cref{eq:ams_slab_bf_edge} yields
    \begin{align*}
        \tilde{A}_{e' e'} \mathbf{1}_{e'} + A_{e' \text{V}_e} \mathbf{1}_{\text{V}_e}
        &= A_{e' e'} \mathbf{1}_{e'} + A_{e' e^{\ast}} \mathbf{1}_{e^{\ast}} + A_{e' \text{V}_e} \mathbf{1}_{\text{V}_e} \\
        &= \begin{bmatrix}
            A_{e' e^{\ast}} & A_{e' e'} & A_{e' \text{V}_{e}}
        \end{bmatrix} \mathbf{1} \\
        &= \mathbf{0},
    \end{align*}
    where the second equality expresses the row sums of $A$ over the rows associated with $e'$, which vanish for piecewise linear finite elements away from the Dirichlet boundary; cf. \cref{re:ams_pou}. Analogously, for a face slab $f'$, using $\Phi^{\text{AM}}_{\text{E}_f} \mathbf{1}_{\text{V}} = \mathbf{1}_{\text{E}_f}$, one obtains
    \begin{align*}
        \tilde{A}_{f' f'} \mathbf{1}_{f'} + A_{f' \text{E}_f} \mathbf{1}_{\text{E}_f} + A_{f' \text{V}_f} \mathbf{1}_{\text{V}_f}
        &= A_{f' f'} \mathbf{1}_{f'} + A_{f' f^{\ast}} \mathbf{1}_{f^{\ast}} + A_{f' \text{E}_f} \mathbf{1}_{\text{E}_f} + A_{f' \text{V}_f} \mathbf{1}_{\text{V}_f} \\
        &= \begin{bmatrix}
            A_{f' f^{\ast}} & A_{f' f} & A_{f' \text{E}_{f}} & A_{f' \text{V}_{f}}
        \end{bmatrix} \mathbf{1} \\
        &= \mathbf{0}.
    \end{align*}
    Finally, since $\Gamma_i \subset \Gamma_i'$, partitioning the rows of $\Phi^{\text{AM}}_{\Gamma_i'}$ into those associated with $\Gamma_i$ and $\Gamma_i' \setminus \Gamma_i$ yields
    \begin{align*}
        \Phi^{\text{AM}}_{\Gamma_i'} \mathbf{1}_{\text{V}} &= \mathbf{1}_{\Gamma_i'} \\
        \implies
        \begin{bmatrix}
            \Phi^{\text{AM}}_{\Gamma_i} \\
            \Phi^{\text{AM}}_{\Gamma_i' \setminus \Gamma_i}
        \end{bmatrix} \mathbf{1}_{\text{V}}
        &=
        \begin{bmatrix}
            \mathbf{1}_{\Gamma_i} \\
            \mathbf{1}_{\Gamma_i' \setminus \Gamma_i}
        \end{bmatrix} \\
        \implies
        \Phi^{\text{AM}}_{\Gamma_i} \mathbf{1}_{\text{V}} &= \mathbf{1}_{\Gamma_i}.
    \end{align*}
\end{proof}

\begin{remark}
    If a suitable algebraically connected slab cannot be found within a given number of layers of nodes, the interface values on the unconnected entity may instead be prescribed by piecewise linear functions or by a purely algebraic alternative, such as the IPOU of RGDSW option 1 in \cref{sec:rgdsw_coarse_space}. We adopt the latter in our implementation to preserve the algebraic nature of the framework.
\end{remark}

\subsection{Framework for algebraic multiscale coarse spaces for elliptic problems} \label{sec:ams_rgdsw_framework}

The procedure described in \cref{sec:ams_basis_functions} was devised for diffusion problems. Its application to the linear elasticity equations requires the careful definition of the RBC problems to handle the components of the nodal dofs \cite{BUCK2013MSFEMLE}. Moreover, changes to \cref{eq:msfem_bf_def} directly impact the interpretation of the reduced diagonal blocks $\tilde{A}_{\text{EE}}$ and $\tilde{A}_{\text{FF}}$; cf. \cite{CASTELLETTO2017AMSLE, SOKOLOVA2019AMSPOROELASTICITY}.

The algebraic multiscale procedure may nonetheless be used to compute an interface partition of unity (IPOU) function that can be employed as part of the RGDSW framework outlined in \cref{sec:rgdsw_coarse_space}. This IPOU function incorporates information about the changes in the material, thus improving the robustness of the coarse space for heterogeneous problems. In addition, by definition of the RGDSW method, the resulting coarse space is guaranteed to contain the null space.

To state the common AM-RGDSW framework, we start by rewriting the AM coarse basis functions for the scalar diffusion problem in terms of the energy-minimizing extension in \cref{eq:energy_min_ext}. Let $\Gamma = \text{V} \cup \text{E} \cup \text{F}$ be the set of interface dofs (in 2D, $\Gamma = \text{E} \cup \text{V}$). From \cref{eq:phi_ams},
\begin{align*}
    \Phi^{\text{AM}}_{\text{I}} &= -A_{\text{II}}^{-1} \left( A_{\text{IV}} \Phi^{\text{AM}}_{\text{V}} + A_{\text{IE}} \Phi^{\text{AM}}_{\text{E}} + A_{\text{IF}} \Phi^{\text{AM}}_{\text{F}} \right) \\
    &= -A_{\text{II}}^{-1}
    \begin{bmatrix}
        A_{\text{IF}} & A_{\text{IE}} & A_{\text{IV}}
    \end{bmatrix}
    \begin{bmatrix}
        \Phi^{\text{AM}}_{\text{F}} \\
        \Phi^{\text{AM}}_{\text{E}} \\
        \Phi^{\text{AM}}_{\text{V}}
    \end{bmatrix}.
\end{align*}
Denoting
\begin{equation*}
    A_{\text{I}\Gamma} =
    \begin{bmatrix}
        A_{\text{IF}} & A_{\text{IE}} & A_{\text{IV}}
    \end{bmatrix}, \,
    \Phi_{\Gamma}^{\text{AM}} = \begin{bmatrix}
        \Phi^{\text{AM}}_{\text{F}} \\
        \Phi^{\text{AM}}_{\text{E}} \\
        \Phi^{\text{AM}}_{\text{V}}
    \end{bmatrix},
\end{equation*}
we may write the AM coarse operator as
\begin{equation*}
    \Phi^{\text{AM}} =
    \begin{bmatrix}
        \Phi_{\text{I}}^{\text{AM}} \\
        \Phi_{\Gamma}^{\text{AM}}
    \end{bmatrix} =
    \begin{bmatrix}
        -A_{\text{II}}^{-1} A_{\text{I}\Gamma} \\
        I_{\Gamma\Gamma}
    \end{bmatrix} \Phi_{\Gamma}^{\text{AM}}.
\end{equation*}

For linear elasticity, we need to reformulate the interface component $\Phi_{\Gamma}^{\text{AM}}$ as the product of an IPOU function and the null space. Let the system matrix be partitioned in terms of the dof components such that
\begin{equation*}
    A =
    \begin{bmatrix}
        A^{(1, 1)} & A^{(1, 2)} & A^{(1, 3)} \\
        A^{(2, 1)} & A^{(2, 2)} & A^{(2, 3)} \\
        A^{(3, 1)} & A^{(3, 2)} & A^{(3, 3)}
    \end{bmatrix},
\end{equation*}
where the blocks $A^{(i, j)}$ correspond to the relation between the dof components $i$ and $j$. In the case of isotropic materials, the diagonal blocks $A^{(i, i)}$ have the same sparsity pattern and depend on the Lamé parameters in the same way. We may select one of them to apply the AM procedure to compute a nodal IPOU function.

Since the elimination procedure in \cref{eq:ams_diagonal_blocks} does not ensure that the terms related to the other dof components are subtracted from the diagonal, we need to enforce the partition of unity by normalization. Let the algebraic multiscale interface function computed using one of the $A^{(i, i)}$ in place of the global system matrix be given by
\begin{equation} \label{eq:ams_ipou}
    \Psi^{\text{AM}}_{\Gamma} =
    \begin{bmatrix}
        \Psi^{\text{AM}}_{\text{F}} \\
        \Psi^{\text{AM}}_{\text{E}} \\
        \Psi^{\text{AM}}_{\text{V}}
    \end{bmatrix} =
    \begin{bmatrix}
        -\tilde{A}^{(i, i)^{-1}}_{\text{FF}} \left( A^{(i, i)}_{\text{FV}} \Psi^{\text{AM}}_{\text{V}} + A^{(i, i)}_{\text{FE}} \Psi^{\text{AM}}_{\text{E}} \right) \\
        -\tilde{A}^{(i, i)^{-1}}_{\text{EE}} A^{(i, i)}_{\text{EV}} \Psi^{\text{AM}}_{\text{V}} \\
        I_{\text{VV}}
    \end{bmatrix}.
\end{equation}
To guarantee that a partition of unity is obtained, we normalize the edge and face values after computation. The algebraic multiscale IPOU function is given by
\begin{equation} \label{eq:ams_ipou_fixed}
    \Psi^{\text{AM}}_{\Gamma} =
    \begin{bmatrix}
        \hat{\Psi}^{\text{AM}}_{\text{F}} \\
        \hat{\Psi}^{\text{AM}}_{\text{E}} \\
        \Psi^{\text{AM}}_{\text{V}}
    \end{bmatrix} =
    \begin{bmatrix}
        N_{\text{F}} \overline{\Psi}^{\text{AM}}_{\text{F}} \\
        N_{\text{E}} \Psi^{\text{AM}}_{\text{E}} \\
        I_{\text{VV}}
    \end{bmatrix},
\end{equation}
where $\overline{\Psi}^{\text{AM}}_{\text{F}}$ is computed as in \cref{eq:ams_ipou} by replacing $\Psi^{\text{AM}}_{\text{E}}$ by $\hat{\Psi}^{\text{AM}}_{\text{E}}$, and $N_{\text{F}}$ and $N_{\text{E}}$ are diagonal scaling matrices defined as
\begin{equation*}
    \begin{aligned}
        N_{\text{F}} & \coloneq \diag{\left( \overline{\Psi}^{\text{AM}}_{\text{F}} \mathbf{1}_{\text{V}} \right)}^{-1}, \\
        N_{\text{E}} & \coloneq \diag{\left( \Psi^{\text{AM}}_{\text{E}} \mathbf{1}_{\text{V}} \right)}^{-1}.
    \end{aligned}
\end{equation*}
Finally, to assemble the interface operator $\Phi_{\Gamma}$, we proceed as in \cref{eq:rgdsw_int_bf} by employing \cref{eq:ams_ipou_fixed} as the IPOU function.

\begin{remark}
    Although the formulation of the IPOU function introduced in this subsection is purely algebraic, the null space of the linear elasticity equations requires geometric information to compute the rotation modes. Nevertheless, these modes may be omitted for a fully algebraic framework, as discussed in \cite{HEINLEINHOCHMUTHKLAWONN2021ALGEBRAICLE} for the RGDSW coarse space. We investigate the impact of removing the rotation modes from the null space in \cref{sec:numerical_results}.
\end{remark}

\section{Numerical results} \label{sec:numerical_results}

In this section, we present numerical results on the robustness and scalability of the AM coarse space for scalar diffusion and linear elasticity problems. We compare the numerical performance of the AM coarse space to the GDSW and RGDSW option 1 algebraic coarse spaces. In all cases, unless stated otherwise, we have used P1 finite elements to discretize the problem and the resulting system of equations was solved using the preconditioned conjugate gradient (PCG) method with a relative stopping criterion $|| r^{(k)} ||_2 / || r^{(0)} ||_2 < 10^{-6}$, where $r^{(0)}$ and $r^{(k)}$ are the initial and $k$-th unpreconditioned residuals, respectively. For the Schwarz method, unless otherwise noted, we employ structured quadrangular and hexahedral domain decompositions in 2D and 3D, respectively, with subdomain, grid, and overlap sizes denoted by $H$, $h$, and $\delta$, respectively. Moreover, the overlapping region is defined algebraically by inspecting the node connectivity in the stiffness matrix.

All methods discussed in this section were implemented in Python using the libraries NumPy \cite{NUMPY2020} version $2.1.1$ and SciPy \cite{SCIPY2020} version $1.14.1$. In addition, the geometric modeling, mesh generation and assembly of the finite element problem were done with NGSolve \cite{NGSOLVEREPORT2014, NGSOLVEPROJECT} version $6.2.2506$.

\subsection{Diffusion: 2D high coefficient inclusions} \label{sec:results_diffusion_high_coeff}

\begin{figure}[t]
  \centering
  \begin{subfigure}[b]{0.495\textwidth}
      \centering
      \includegraphics[width=0.7\textwidth]{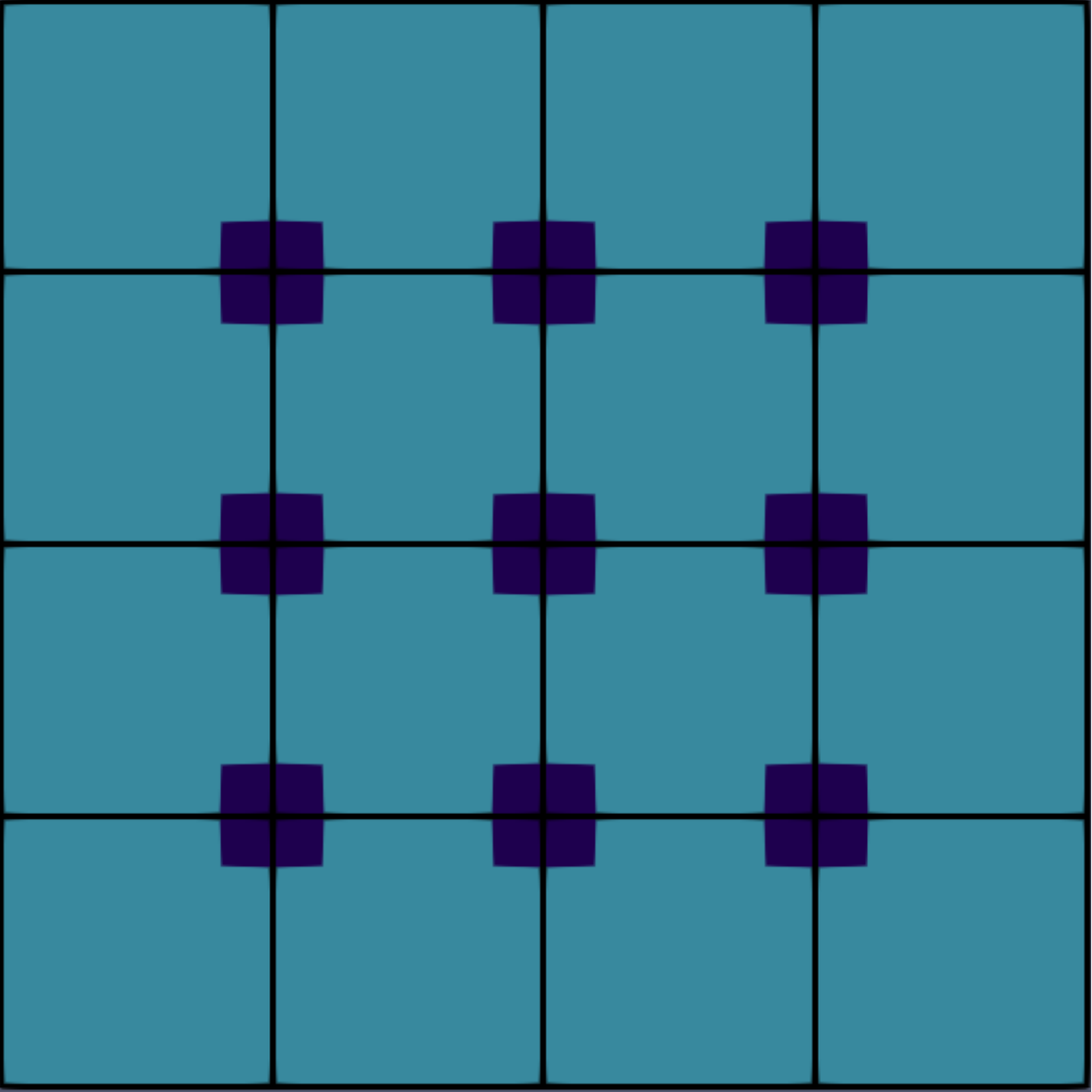}
      \caption{}
      \label{fig:vertex_inclusions}
  \end{subfigure}
  \hfill
  \begin{subfigure}[b]{0.495\textwidth}
      \centering
      \includegraphics[width=0.7\textwidth]{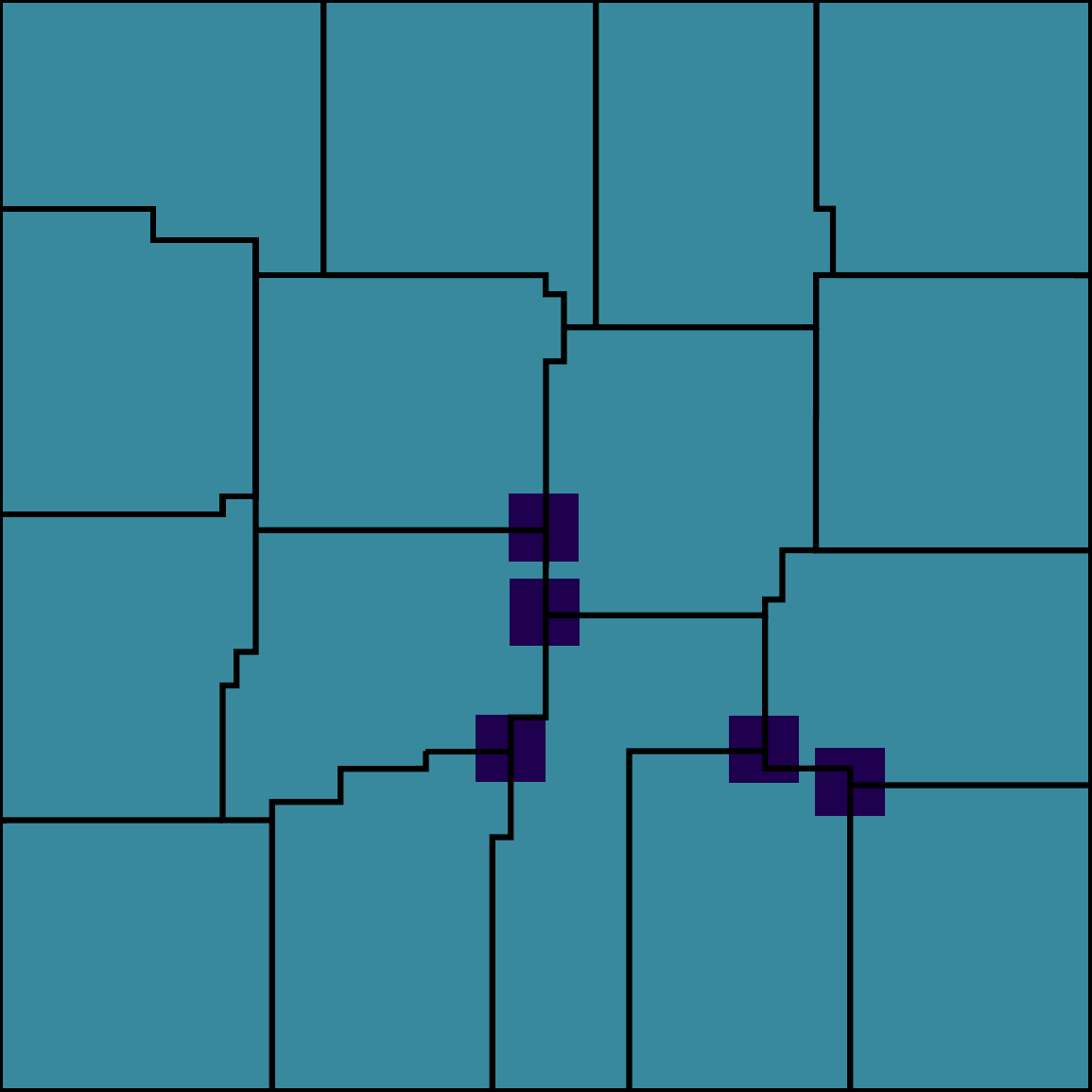}
      \caption{}
      \label{fig:vertex_inclusions_unstruct}
  \end{subfigure}
  \hfill
  \begin{subfigure}[b]{0.495\textwidth}
      \centering
      \includegraphics[width=0.7\textwidth]{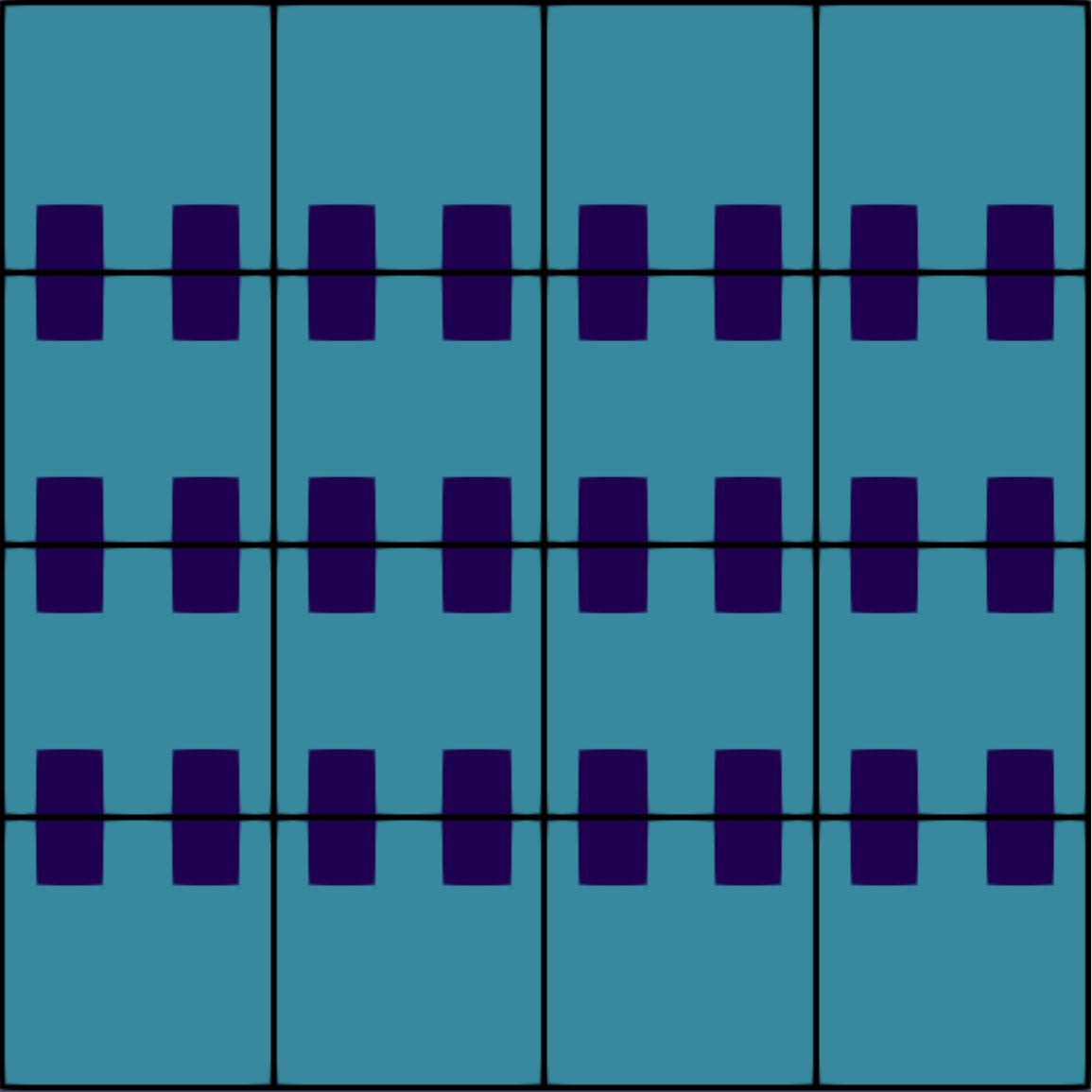}
      \caption{}
      \label{fig:edge_inclusions}
  \end{subfigure}
  \hfill
  \begin{subfigure}[b]{0.495\textwidth}
      \centering
      \includegraphics[width=0.7\textwidth]{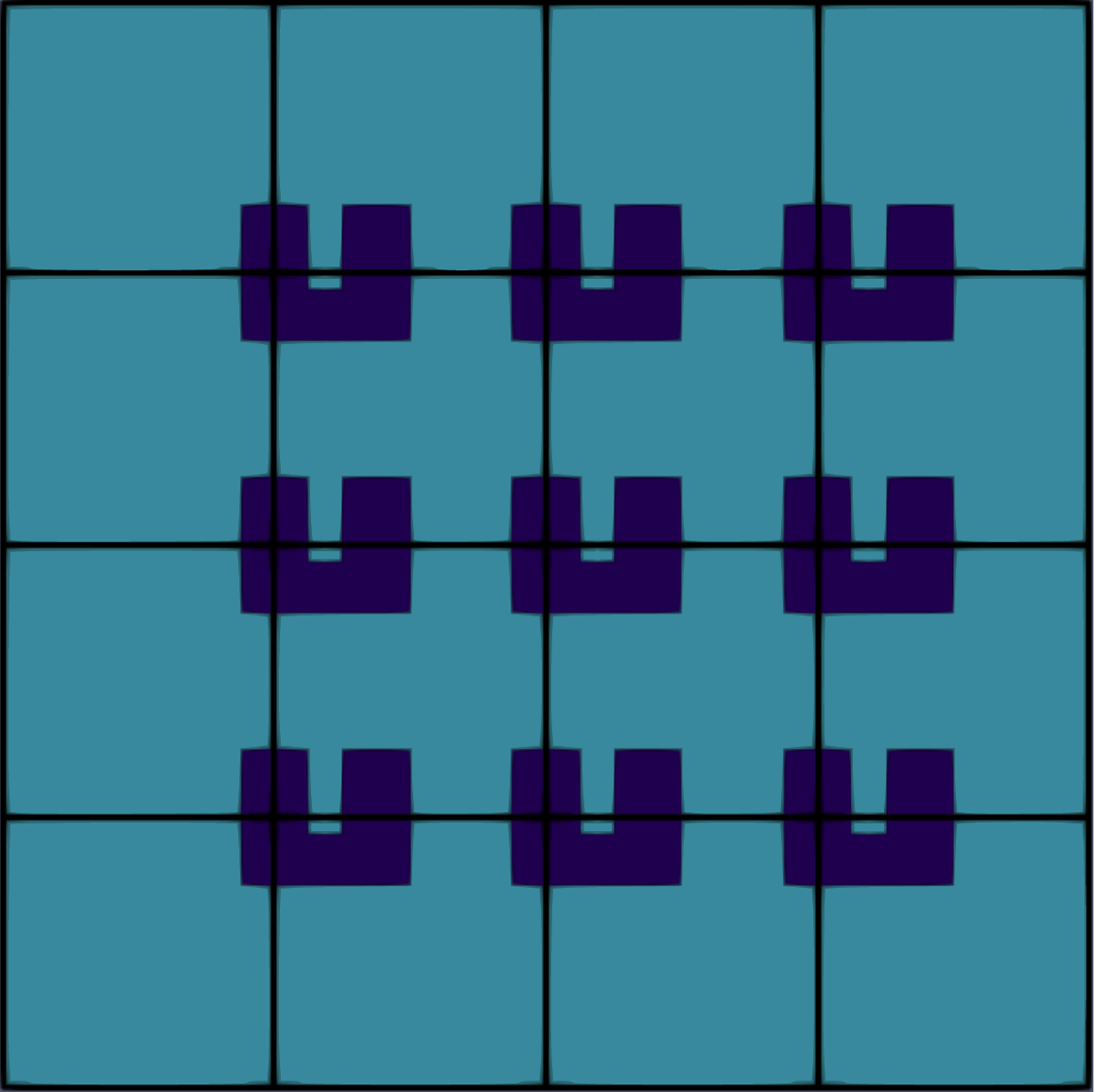}
      \caption{}
      \label{fig:u_shaped_inclusions}
  \end{subfigure}
  \caption{Coefficient distribution with high coefficient inclusions used in the examples. In the darker areas, $\alpha (x) = \alpha_{\text{max}} \geq 1$ while $\alpha (x) = 1$ elsewhere. The black grid lines correspond to the non-overlapping domain decomposition. For the structured partitions in (a), (c) and (d), $1/H = 4$ and $H/h = 16$.}
  \label{fig:high_coeff_inclusions}
\end{figure}

As a first example, we consider the heterogeneous scalar diffusion problem with coefficient distributions containing high coefficient inclusions, as shown in \cref{fig:high_coeff_inclusions}. A constant forcing term $f = 1$ is chosen for all cases.

The weak scalability results for the coefficient distribution in \cref{fig:vertex_inclusions} with $\alpha_{\text{max}} = 10^6$ are shown in \cref{fig:diffusion_vertex_inclusions_weak_scalability} for $H/h = 16$ and two layers of algebraic overlap. The AM coarse space is robust with respect to the contrast while the other options are not. In \cref{fig:diffusion_vertex_inclusions_weak_scalability_cond_num}, it can be seen that the condition number for GDSW and RGDSW is of the same order of magnitude as the contrast, whereas AM is independent of it. The same is observed for unstructured subdomains, as in \cref{fig:vertex_inclusions_unstruct}. The iteration count and condition number are shown in \cref{fig:diffusion_vertex_inclusions_weak_scalability_unstruct}. In this case, the high coefficient inclusions are added to coarse nodes that are far enough apart such that the inclusions do not cover an entire edge of the subdomain. This is done so that the assumptions in \cref{sec:conv_analysis_ams} are satisfied.

The effect of the contrast is further detailed in \cref{tab:diffusion_high_coeff_contrast_study} and \cref{tab:diffusion_high_coeff_contrast_study_unstruct}, where the problem is solved for increasing $\alpha_{\text{max}}$, the value of the coefficient function in the inclusions. For the AM coarse space, the iteration count, condition number and coarse space robustness indicator remain roughly constant. This is in line with the theory discussed in \cref{sec:conv_analysis_ams}. From \cref{eq:cond_num_bound_msfem}, for fixed values of $H$, $h$, and $\delta$, the upper bound on the condition number depends only on $\gamma(\alpha)$. In addition, since the coefficient distribution satisfies the assumptions stipulated in \cref{sec:conv_analysis_ams}, the coarse space robustness indicator can be bounded independently of the coefficient contrast. Therefore, for a fixed ratio $H/\delta$, the condition number is bounded from above by a constant.

\begin{figure}[t]
    \centering
    \begin{subfigure}[b]{0.495\textwidth}
        \includegraphics[width=\textwidth]{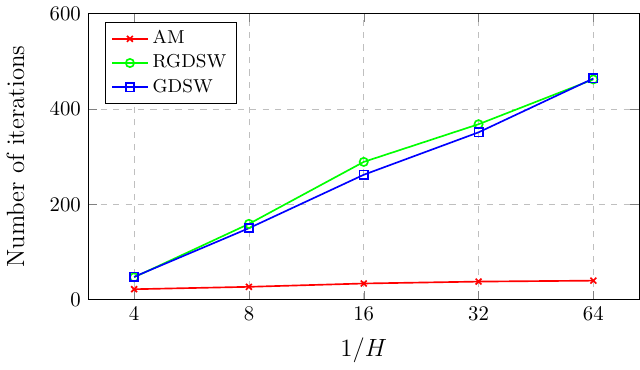}
        \caption{}
        \label{fig:diffusion_vertex_inclusions_weak_scalability_iter}
    \end{subfigure}
    \hfil
    \begin{subfigure}[b]{0.495\textwidth}
        \includegraphics[width=\textwidth]{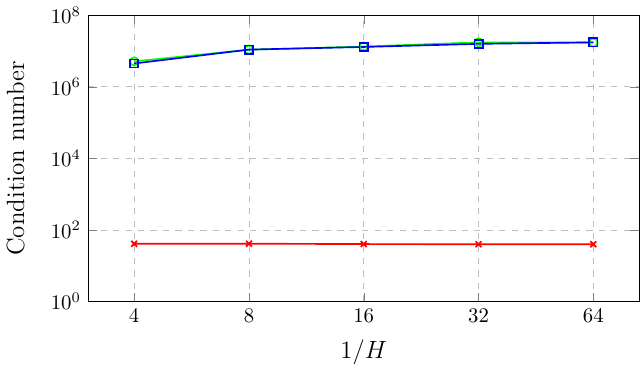}
        \caption{}
        \label{fig:diffusion_vertex_inclusions_weak_scalability_cond_num}
    \end{subfigure}
    \caption{Numerical scalability of the AM coarse space compared to the GDSW and RGDSW methods for the coefficient distribution in \cref{fig:vertex_inclusions}. Results for $H/h = 16$ and $\delta = 2h$. \textbf{Left:} number of iterations; \textbf{right:} condition number estimate.}
    \label{fig:diffusion_vertex_inclusions_weak_scalability}
\end{figure}

\begin{figure}[t]
    \centering
    \begin{subfigure}[b]{0.495\textwidth}
        \includegraphics[width=\textwidth]{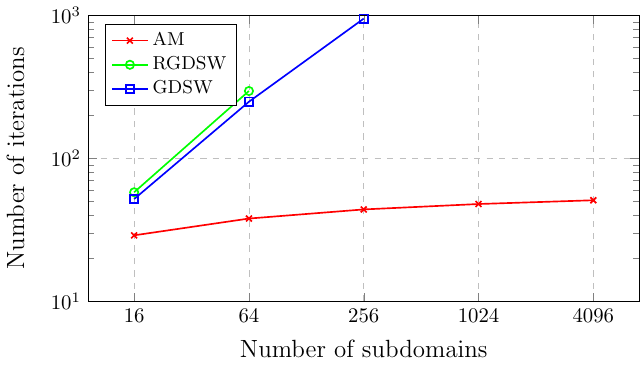}
        \caption{}
        \label{fig:diffusion_vertex_inclusions_weak_scalability_unstruct_iter}
    \end{subfigure}
    \hfil
    \begin{subfigure}[b]{0.495\textwidth}
        \includegraphics[width=\textwidth]{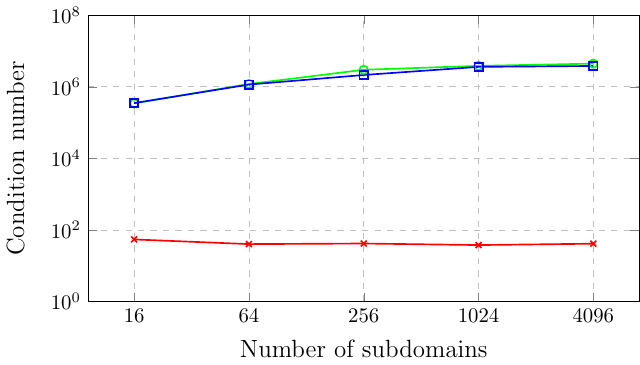}
        \caption{}
        \label{fig:diffusion_vertex_inclusions_weak_scalability_unstruct_cond_num}
    \end{subfigure}
    \caption{Numerical scalability of the AM coarse space compared to the GDSW and RGDSW methods for the coefficient distribution in \cref{fig:vertex_inclusions_unstruct} on unstructured domain decompositions. Results for $\delta = 2h$. \textbf{Left:} number of iterations; \textbf{right:} condition number estimate. GDSW and RGDSW did not converge within 1000 iterations for more than 128 and 64 subdomains, respectively.}
    \label{fig:diffusion_vertex_inclusions_weak_scalability_unstruct}
\end{figure}

\begin{table}[t]
  \caption{Number of iterations (it.), condition number estimate ($\kappa$), and coarse space robustness indicator ($\gamma(\alpha)$) for each algebraic coarse space option and different coefficient contrasts in the coefficient distribution in \cref{fig:vertex_inclusions}. Here, $1/H = 8$, $1/h = 128$, and $\delta = 2h$.}
  \label{tab:diffusion_high_coeff_contrast_study}

  \footnotesize
  \begin{subtable}{0.475\linewidth}
    \begin{tabular*}{\textwidth}{@{\extracolsep{\fill}}
      >{\RaggedRight}p{2cm} 
      *{3}{c} }
      Coarse space & it. & $\kappa$ & $\gamma(\alpha)$ \\
      \midrule
      AM & 22 & 20.35 & 11.77 \\
      RGDSW & 28 & 32.04 & 13.81 \\
      GDSW & 25 & 20.68 & 9.77 \\
    \end{tabular*}
    \caption{$\alpha_{\text{max}} = 1$}
  \end{subtable}
  \hfill
  \begin{subtable}{0.475\linewidth}
    \begin{tabular*}{\textwidth}{@{\extracolsep{\fill}}
      >{\RaggedRight}p{2cm} 
      *{3}{c} }
      Coarse space & it. & $\kappa$ & $\gamma(\alpha)$ \\
      \midrule
      AM & 26 & 40.51 & 12.63 \\
      RGDSW & 66 & $1.19 \cdot 10^3$ & $3.93 \cdot 10^2$ \\
      GDSW & 65 & $1.15 \cdot 10^3$ & $5.84 \cdot 10^2$ \\
    \end{tabular*}
    \caption{$\alpha_{\text{max}} = 10^2$}
  \end{subtable}
  \hfill
  \begin{subtable}{0.475\linewidth}
    \begin{tabular*}{\textwidth}{@{\extracolsep{\fill}}
      >{\RaggedRight}p{2cm} 
      *{3}{c} }
      Coarse space & it. & $\kappa$ & $\gamma(\alpha)$ \\
      \midrule
      AM & 27 & 43.92 & 12.65 \\
      RGDSW & 110 & $1.08 \cdot 10^5$ & $3.90 \cdot 10^4$ \\
      GDSW & 108 & $9.76 \cdot 10^4$ & $5.80 \cdot 10^4$ \\
    \end{tabular*}
    \caption{$\alpha_{\text{max}} = 10^4$}
  \end{subtable}
  \hfill
  \begin{subtable}{0.475\linewidth}
    \begin{tabular*}{\textwidth}{@{\extracolsep{\fill}}
      >{\RaggedRight}p{2cm} 
      *{3}{c} }
      Coarse space & it. & $\kappa$ & $\gamma(\alpha)$ \\
      \midrule
      AM & 27 & 41.23 & 12.65 \\
      RGDSW & 156 & $1.10 \cdot 10^7$ & $3.90 \cdot 10^6$ \\
      GDSW & 150 & $1.06 \cdot 10^7$ & $5.80 \cdot 10^6$ \\
    \end{tabular*}
    \caption{$\alpha_{\text{max}} = 10^6$}
  \end{subtable}
\end{table}

\begin{table}[t]
  \caption{Number of iterations (it.), condition number estimate ($\kappa$), and coarse space robustness indicator ($\gamma(\alpha)$) for each algebraic coarse space option and different coefficient contrasts in the coefficient distribution in \cref{fig:vertex_inclusions_unstruct}. Here, $1/h = 128$, $\delta = 2h$, and 128 subdomains are employed.}
  \label{tab:diffusion_high_coeff_contrast_study_unstruct}

  \footnotesize
  \begin{subtable}{0.475\linewidth}
    \begin{tabular*}{\textwidth}{@{\extracolsep{\fill}}
      >{\RaggedRight}p{2cm} 
      *{3}{c} }
      Coarse space & it. & $\kappa$ & $\gamma(\alpha)$ \\
      \midrule
      AM & 27 & 17.64 & 13.28 \\
      RGDSW & 31 & 22.13 & 13.70 \\
      GDSW & 30 & 19.76 & 11.21 \\
    \end{tabular*}
    \caption{$\alpha_{\text{max}} = 1$}
  \end{subtable}
  \hfill
  \begin{subtable}{0.475\linewidth}
    \begin{tabular*}{\textwidth}{@{\extracolsep{\fill}}
      >{\RaggedRight}p{2cm} 
      *{3}{c} }
      Coarse space & it. & $\kappa$ & $\gamma(\alpha)$ \\
      \midrule
      AM & 35 & 28.91 & 13.98 \\
      RGDSW & 74 & $1.15 \cdot 10^2$ & $2.70 \cdot 10^2$ \\
      GDSW & 69 & $1.09 \cdot 10^2$ & $4.83 \cdot 10^2$ \\
    \end{tabular*}
    \caption{$\alpha_{\text{max}} = 10^2$}
  \end{subtable}
  \hfill
  \begin{subtable}{0.475\linewidth}
    \begin{tabular*}{\textwidth}{@{\extracolsep{\fill}}
      >{\RaggedRight}p{2cm} 
      *{3}{c} }
      Coarse space & it. & $\kappa$ & $\gamma(\alpha)$ \\
      \midrule
      AM & 38 & 40.03 & 14.09 \\
      RGDSW & 186 & $1.20 \cdot 10^4$ & $2.67 \cdot 10^4$ \\
      GDSW & 164 & $1.05 \cdot 10^4$ & $4.79 \cdot 10^4$ \\
    \end{tabular*}
    \caption{$\alpha_{\text{max}} = 10^4$}
  \end{subtable}
  \hfill
  \begin{subtable}{0.475\linewidth}
    \begin{tabular*}{\textwidth}{@{\extracolsep{\fill}}
      >{\RaggedRight}p{2cm} 
      *{3}{c} }
      Coarse space & it. & $\kappa$ & $\gamma(\alpha)$ \\
      \midrule
      AM & 38 & 40.20 & 14.10 \\
      RGDSW & 296 & $1.20 \cdot 10^6$ & $2.67 \cdot 10^6$ \\
      GDSW & 249 & $1.16 \cdot 10^6$ & $4.79 \cdot 10^6$ \\
    \end{tabular*}
    \caption{$\alpha_{\text{max}} = 10^6$}
  \end{subtable}
\end{table}

We consider now the coefficient distribution in \cref{fig:edge_inclusions} with $\alpha_{\text{max}} = 10^6$. The weak scalability results are shown in \cref{fig:diffusion_edge_inclusions_weak_scalability}. For this distribution, none of the coarse space options investigated are fully robust with respect to the contrast. In all cases, the number of iterations increases as more subdomains are employed, and the condition number is of the same order of magnitude as the contrast. However, the increase in the number of iterations of AM and GDSW is not as sharp as that of RGDSW. By inspecting the spectra of eigenvalues in \cref{fig:diffusion_edge_inclusions_eigval_spectra}, it can be noticed that, for AM and GDSW, the eigenvalues are grouped in two clusters. For RGDSW, the cluster around 1 tends to spread. This relates to its worse iterative performance, since the CG method can resolve the highly clustered eigenvalues well; see \cite{SOLIMAN2026PCGBOUND} for a discussion on the iterative performance of PCG for clustered spectra.

\begin{figure}[t]
    \centering
    \begin{subfigure}[b]{0.495\textwidth}
        \includegraphics[width=\textwidth]{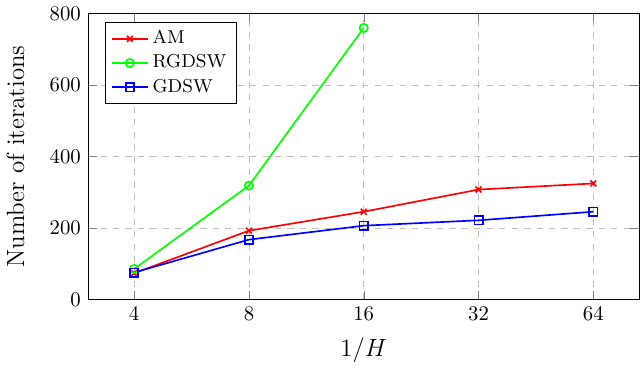}
        \caption{}
        \label{fig:diffusion_edge_inclusions_weak_scalability_iter}
    \end{subfigure}
    \hfil
    \begin{subfigure}[b]{0.495\textwidth}
        \includegraphics[width=\textwidth]{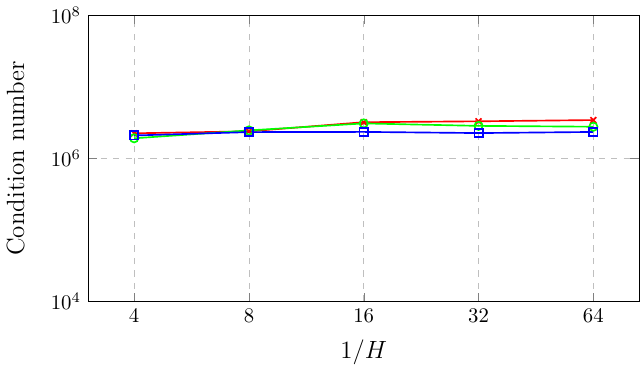}
        \caption{}
        \label{fig:diffusion_edge_inclusions_weak_scalability_cond_num}
    \end{subfigure}
    \caption{Numerical scalability of the AM coarse space compared to the GDSW and RGDSW methods for the coefficient distribution in \cref{fig:edge_inclusions}. Results for $H/h = 16$ and $\delta = 2h$. For $1/H > 16$, RGDSW did not converge within 1000 iterations. \textbf{Left:} number of iterations; \textbf{right:} condition number estimate.}
    \label{fig:diffusion_edge_inclusions_weak_scalability}
\end{figure}

\begin{figure}[t]
  \centering
  \begin{subfigure}[b]{0.495\textwidth}
    \includegraphics[width=\textwidth]{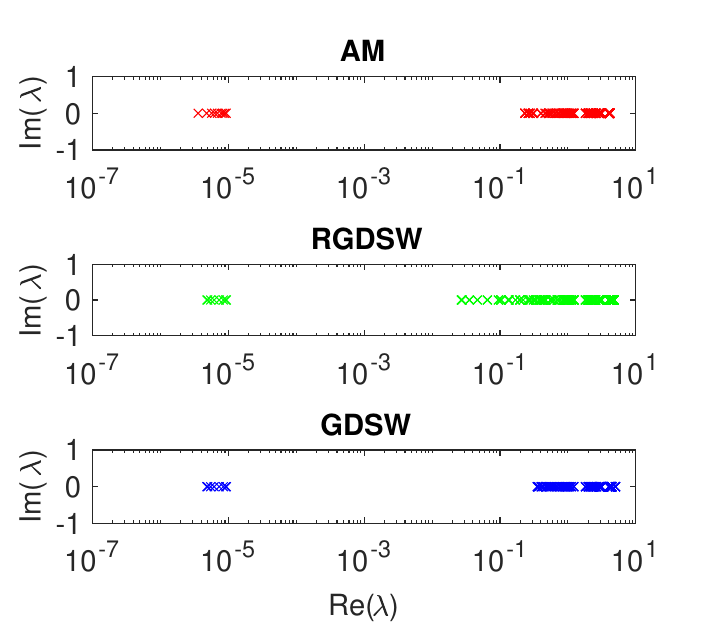}
    \caption{$1/H = 8$}
    \label{fig:diffusion_edge_inclusions_eigval_spectra_8x8}
  \end{subfigure}
  \hfil
  \begin{subfigure}[b]{0.495\textwidth}
    \includegraphics[width=\textwidth]{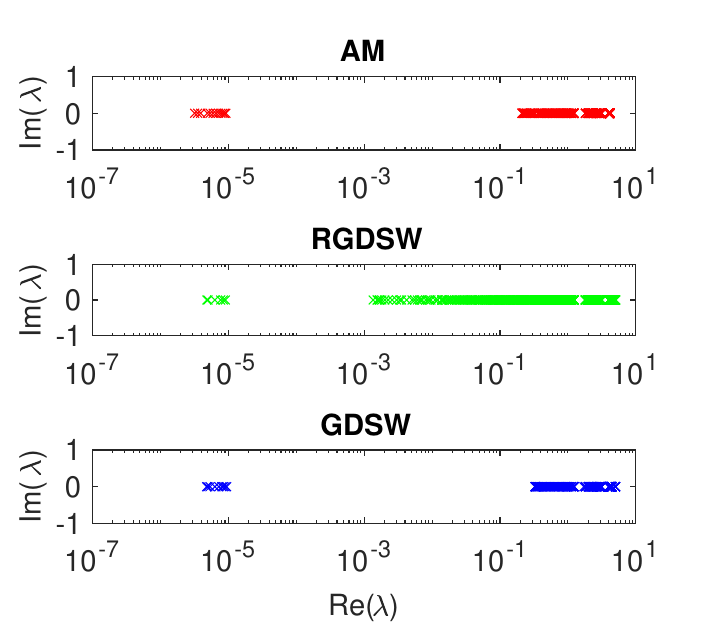}
    \caption{$1/H = 32$}
    \label{fig:diffusion_edge_inclusions_eigval_spectra_32x32}
  \end{subfigure}
  \caption{Spectra of eigenvalues of the Schwarz operator with each algebraic coarse space option for the coefficient distribution in \cref{fig:edge_inclusions} applied to the diffusion problem. The eigenvalues are computed from the coefficients of the Lanczos algorithm associated with the CG iterations. \textbf{Left:} $1/H = 8$; \textbf{right:} $1/H = 32$.}
  \label{fig:diffusion_edge_inclusions_eigval_spectra}
\end{figure}

Finally, to illustrate the impact of the slab entities on the robustness of the AM coarse space, we consider the coefficient distribution in \cref{fig:u_shaped_inclusions} again with $\alpha_{\text{max}} = 10^6$. The weak scalability results are shown in \cref{fig:diffusion_u_shaped_inclusions_weak_scalability}. Each edge entity was extended by one layer of nodes. For this distribution, the AM coarse space without the slab entities is not robust with respect to the contrast. In fact, for $1/H > 16$, it did not converge within 1000 iterations. However, when the slab entities are included, the AM coarse space is robust with respect to the contrast. This is because the modified RBC problem solved on the edges is capable of resolving the jumps on the edge while the usual AM coarse space cannot.

\begin{figure}[t]
    \centering
    \begin{subfigure}[b]{0.495\textwidth}
        \includegraphics[width=\textwidth]{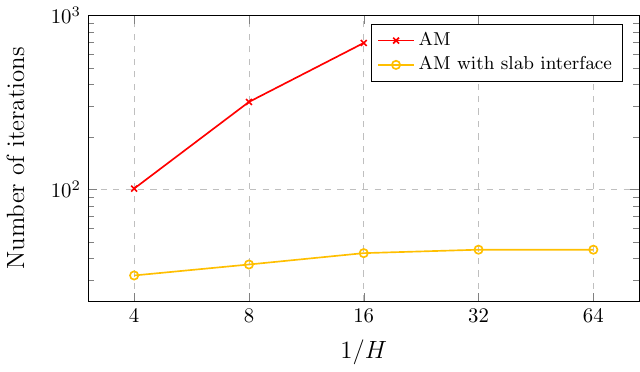}
        \caption{}
        \label{fig:diffusion_u_shaped_inclusions_weak_scalability_iter}
    \end{subfigure}
    \hfil
    \begin{subfigure}[b]{0.495\textwidth}
        \includegraphics[width=\textwidth]{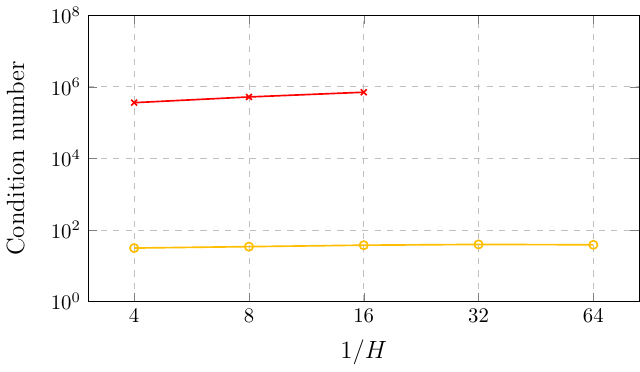}
        \caption{}
        \label{fig:diffusion_u_shaped_inclusions_weak_scalability_cond_num}
    \end{subfigure}
    \caption{Comparison of the numerical scalability of the AM coarse space with and without the slab interface entities for the coefficient distribution in \cref{fig:u_shaped_inclusions} applied to the scalar diffusion problem. Results for $H/h = 16$ and $\delta = 2h$. For $1/H > 16$, AM without the slab entities did not converge within 1000 iterations. \textbf{Left:} number of iterations; \textbf{right:} condition number estimate.}
    \label{fig:diffusion_u_shaped_inclusions_weak_scalability}
\end{figure}

\subsection{Diffusion: 3D log-normal distribution} \label{sec:results_diffusion_geostat}

\begin{figure}[t]
    \centering
    \begin{subfigure}[b]{0.495\textwidth}
        \includegraphics[width=\textwidth]{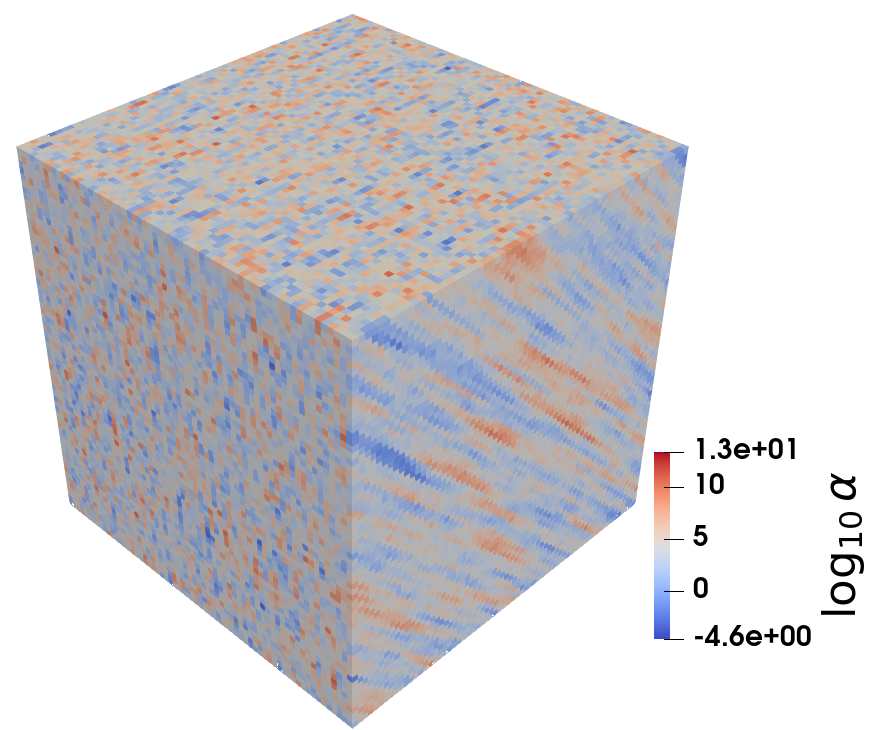}
        \caption{}
        \label{fig:geostat_coeff_layered}
    \end{subfigure}
    \hfill
    \begin{subfigure}[b]{0.495\textwidth}
        \includegraphics[width=\textwidth]{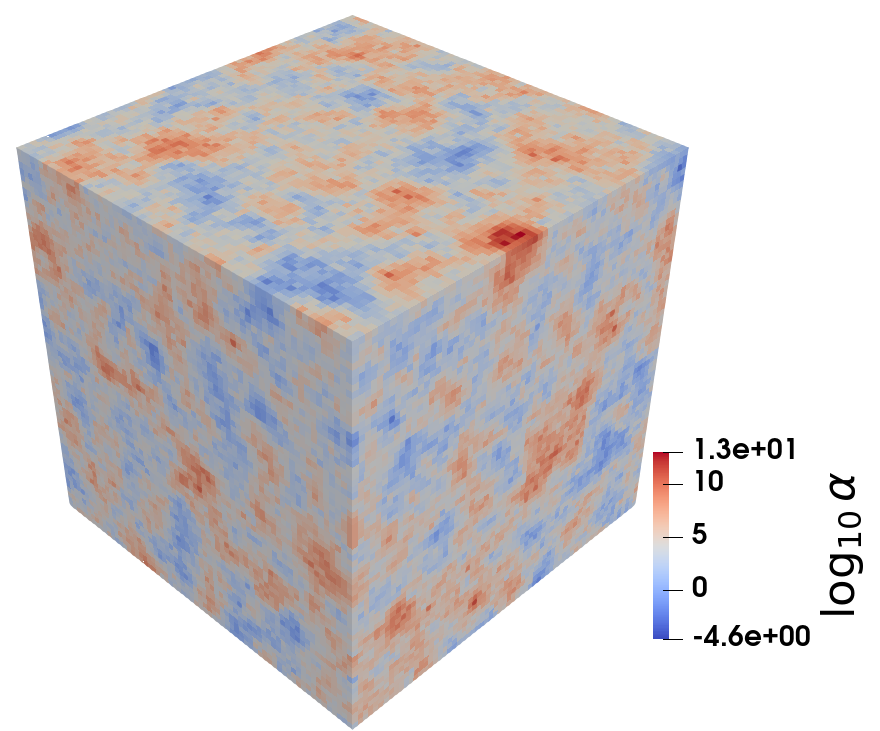}
        \caption{}
        \label{fig:geostat_coeff_patchy}
    \end{subfigure}
    \caption{Realizations of the log-normal coefficient distribution used for the numerical examples in subsections \labelcref{sec:results_diffusion_geostat} and \labelcref{sec:ams_comparison}. \textbf{Left:} coefficient distributed in rotated layers; \textbf{right:} coefficient distributed in heterogeneous patches.}
    \label{fig:geostat_coeff}
\end{figure}

Next, we investigate a 3D problem based on a set of 20 equiprobable realizations of a log-normal coefficient distribution with spherical variograms. Each realization was generated using sequential Gaussian simulation with mean and variance of $\ln{\alpha}$ equal to $-1$ and 4, respectively. An example is illustrated in \cref{fig:geostat_coeff_layered}. The contrast $\alpha_{\text{max}} / \alpha_{\text{min}}$ is of the order of $10^{17}$. Moreover, the coefficient is structured in rotated layers, making the heterogeneities cross the subdomains. As in the previous example, we adopt a constant forcing term $f = 1$.

Although the bound on the coarse space robustness indicator does not apply in this case, the upper bound on the condition number in \cref{eq:cond_num_bound_msfem} is still valid. The numerical results in \cref{tab:diffusion_lognormal_parametric_study} corroborate the theory on the two-level OAS preconditioner. The average iteration count tends to deteriorate for larger $H$ and fixed $\delta$, and improve for larger $\delta$ and fixed $H$. Moreover, the performance of the AM coarse space is comparable to GDSW, with both performing better than RGDSW. This can be more clearly seen in \cref{fig:diffusion_lognormal_parametric_study_iter}. The advantage of using AM instead of GDSW in this scenario is that the multiscale coarse space is considerably smaller, since only vertex-centered coarse basis functions are used. As shown in \cref{tab:diffusion_lognormal_coarse_space_size}, the GDSW coarse space is up to 10 times larger than AM.

\begin{table}[t]
  \caption{Average number of iterations and condition number over 20 equiprobable realizations of a log-normal coefficient distribution for each algebraic coarse space with $h = 1/64$.}
  \label{tab:diffusion_lognormal_parametric_study}

  \centering
  \footnotesize
  \setlength\tabcolsep{0pt}

  \begin{subtable}{0.75\linewidth}
    \begin{tabular*}{\textwidth}{@{\extracolsep{\fill}}*{7}{c} }
      & \multicolumn{3}{c}{Iterations} & \multicolumn{3}{c}{Condition number} \\ 
      \cmidrule{2-4} \cmidrule{5-7}
      \diagbox{$\delta$}{$H$} & $4h$ & $8h$ & $16h$ & $4h$ & $8h$ & $16h$ \\
      \midrule
      $1h$ & 142 & 258 & 350 & $9.42 \cdot 10^2$ & $2.81 \cdot 10^3$ & $7.78 \cdot 10^3$ \\
      $2h$ & 110 & 187 & 243 & $5.23 \cdot 10^2$ & $1.24 \cdot 10^3$ & $3.85 \cdot 10^3$ \\
      $4h$ & 112 & 127 & 130 & $6.27 \cdot 10^2$ & $7.54 \cdot 10^2$ & $1.35 \cdot 10^3$ \\
    \end{tabular*}
    \caption{AM} \label{tab:diffusion_lognormal_parametric_study_ams}
  \end{subtable}
  \hfill
  \begin{subtable}{0.75\linewidth}
    \begin{tabular*}{\textwidth}{@{\extracolsep{\fill}}*{7}{c} }
      & \multicolumn{3}{c}{Iterations} & \multicolumn{3}{c}{Condition number} \\ 
      \cmidrule{2-4} \cmidrule{5-7}
      \diagbox{$\delta$}{$H$} & $4h$ & $8h$ & $16h$ & $4h$ & $8h$ & $16h$ \\
      \midrule
      $1h$ & 371 & 476 & 425 & $6.73 \cdot 10^3$ & $1.20 \cdot 10^4$ & $1.85 \cdot 10^4$ \\
      $2h$ & 292 & 338 & 285 & $4.02 \cdot 10^3$ & $5.23 \cdot 10^3$ & $8.36 \cdot 10^3$ \\
      $4h$ & 291 & 203 & 148 & $3.24 \cdot 10^3$ & $2.32 \cdot 10^3$ & $2.13 \cdot 10^3$ \\
    \end{tabular*}
    \caption{RGDSW} \label{tab:diffusion_lognormal_parametric_study_rgdsw}
  \end{subtable}
  \hfill
  \begin{subtable}{0.75\linewidth}
    \begin{tabular*}{\textwidth}{@{\extracolsep{\fill}}*{7}{c} }
      & \multicolumn{3}{c}{Iterations} & \multicolumn{3}{c}{Condition number} \\ 
      \cmidrule{2-4} \cmidrule{5-7}
      \diagbox{$\delta$}{$H$} & $4h$ & $8h$ & $16h$ & $4h$ & $8h$ & $16h$ \\
      \midrule
      $1h$ & 185 & 296 & 339 & $1.97 \cdot 10^3$ & $5.48 \cdot 10^3$ & $5.86 \cdot 10^3$ \\
      $2h$ & 149 & 206 & 237 & $1.19 \cdot 10^3$ & $2.59 \cdot 10^3$ & $3.19 \cdot 10^3$ \\
      $4h$ & 132 & 145 & 135 & $9.80 \cdot 10^2$ & $1.07 \cdot 10^3$ & $1.15 \cdot 10^3$ \\
    \end{tabular*}
    \caption{GDSW} \label{tab:diffusion_lognormal_parametric_study_gdsw}
  \end{subtable}
\end{table}

\begin{table}[t]
  \caption{Comparison of coarse space sizes for different subdomain sizes $H$.}
  \label{tab:diffusion_lognormal_coarse_space_size}

  \centering
  \footnotesize
  \setlength\tabcolsep{0pt}

  \begin{tabular*}{0.5\textwidth}{@{\extracolsep{\fill}}
    >{\RaggedRight}p{2cm} 
    *{3}{c} }
    & \multicolumn{3}{c}{$H$} \\
    \cmidrule{2-4}
    Coarse space & $4h$ & $8h$ & $16h$ \\
    \midrule
    AM & 3375 & 343 & 27 \\
    GDSW & 25695 & 2863 & 279 \\
  \end{tabular*}
\end{table}

\begin{figure}[t]
    \centering
    \includegraphics[height=0.425\textwidth]{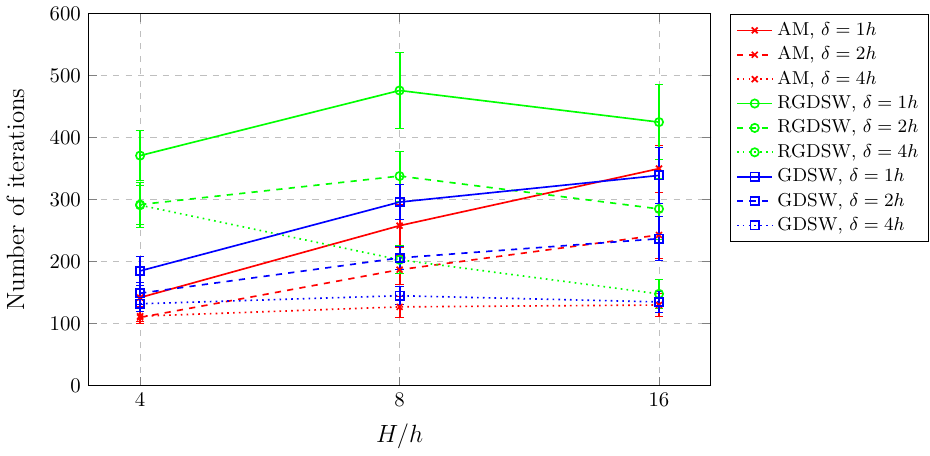}
    \caption{Number of iterations and condition number over 20 equiprobable realizations of a log-normal coefficient distribution for each algebraic coarse space with $h = 1/64$. The plot lines indicate the average and the error bars the standard deviation.}
    \label{fig:diffusion_lognormal_parametric_study_iter}
\end{figure}

\subsection{Linear elasticity: 2D homogeneous problem} \label{sec:results_le_homogeneous}

We now turn our attention to a homogeneous linear elasticity problem as an initial examination of the application of the AM coarse space in this context. The Young modulus is kept constant at $E = 1$, while different values of the Poisson ratio are employed. The body force term is set to $\mathbf{f} = \mathbf{0}$ and the inhomogeneous Dirichlet boundary condition $\mathbf{u}_{\partial \Omega} = (-y,\, x)^{\top}$ is set on the whole domain boundary. The latter corresponds to the linearized rotation rigid body mode.

In \cref{fig:le_homogeneous_weak_scalability_iterations,fig:le_homogeneous_weak_scalability_cond_num}, a weak scalability study is presented for $H/h = 8$, two layers of algebraic overlap and Poisson ratio values increasingly closer to the incompressibility limit $\nu = 0.5$. The problem was solved for 50 different random initial guesses to take into account any possible effects of the initial guess on convergence. The performance of the AM coarse space is comparable to GDSW and RGDSW, scaling well with the number of subdomains for a fixed local subproblem size. We have also considered the case where the rotation modes are omitted from the null space. This allows the coarse basis functions to be computed in a fully algebraic manner. In this setup, the coarse space is still scalable, with an iteration count closer to RGDSW. As $\nu \to 0.5$, the number of iterations and the condition number tend to increase, but all approaches are still scalable.

Next, in \cref{tab:linear_elasticity_homogeneous_parametric_study}, we investigate the behavior of the two-level OAS preconditioner with the AM coarse space for different combinations of the parameters $H$ and $\delta$. For this discussion, $\nu = 0.3$ and the problem is solved with a zero initial guess. From the theory on Schwarz methods \cite{TOSELLI2005DDBOOK}, the upper bound on the condition number of the two-level OAS operator is of the form
\begin{equation} \label{eq:two_level_oas_cond_num_bound_le}
  \kappa \left( M_{\text{OAS-2}}^{-1} A \right) \leq C \left( 1 + \frac{H}{\delta} \right),
\end{equation}
where $C$ is a constant independent of $H$, $h$ and $\delta$. Although we did not develop a specific bound for the AM coarse space, we may still verify that the dependence on $H$ and $\delta$ is the same.

If the entire null space is used, as seen in \cref{tab:linear_elasticity_homogeneous_parametric_study_ams}, the number of iterations and the condition number tend to worsen as $H$ increases and improve for larger $\delta$. Additionally, for fixed $H / \delta$, the iteration count and the condition number remain approximately constant, as expected from \cref{eq:two_level_oas_cond_num_bound_le}. The same was observed when the rotation modes are neglected. For the problem considered in this section, the translation modes appear to be sufficient to provide a robust and scalable preconditioner while allowing for a fully algebraic construction of the coarse basis functions.

\begin{figure}[t]
  \centering
  \begin{subfigure}[b]{0.495\textwidth}
    \centering
    \includegraphics[width=\textwidth]{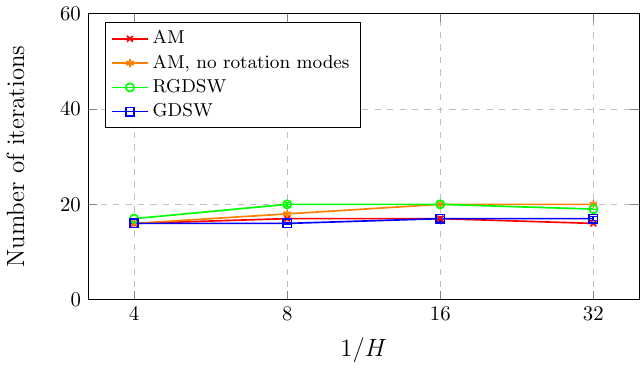}
    \caption{$\nu = 0.3$}
    \label{fig:le_homogeneous_nu_03_weak_scalability_iterations}
  \end{subfigure}
  \hfill
  \begin{subfigure}[b]{0.495\textwidth}
    \centering
    \includegraphics[width=\textwidth]{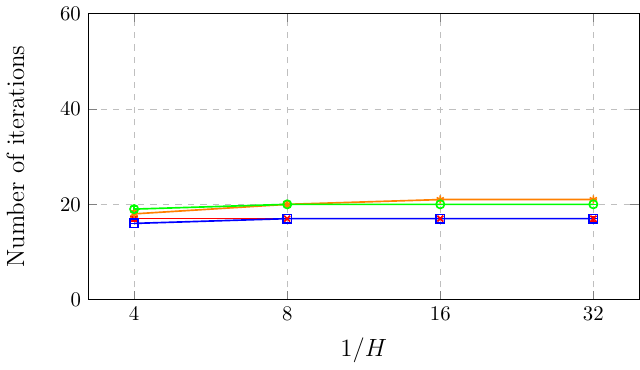}
    \caption{$\nu = 0.4$}
    \label{fig:le_homogeneous_nu_04_weak_scalability_iterations}
  \end{subfigure}
  \hfill
  \begin{subfigure}[b]{0.495\textwidth}
    \centering
    \includegraphics[width=\textwidth]{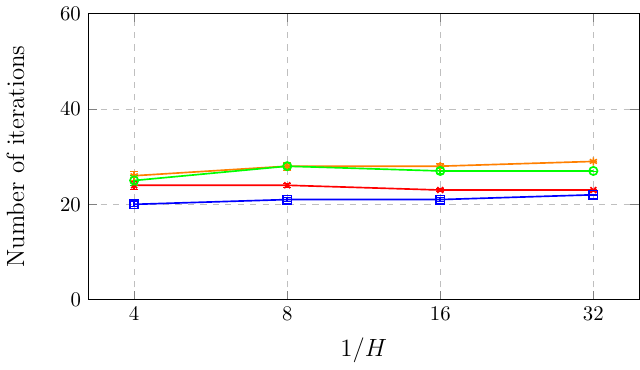}
    \caption{$\nu = 0.49$}
    \label{fig:le_homogeneous_nu_049_weak_scalability_iterations}
  \end{subfigure}
  \hfill
  \begin{subfigure}[b]{0.495\textwidth}
    \centering
    \includegraphics[width=\textwidth]{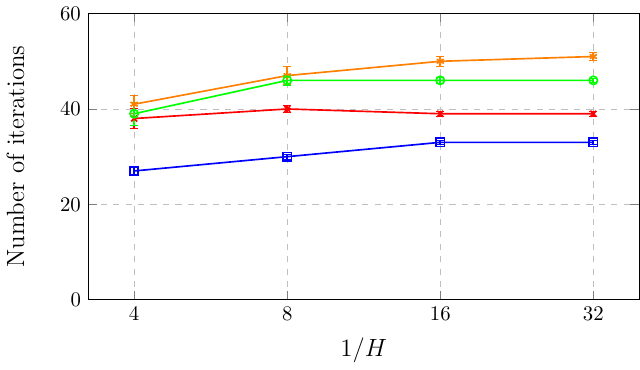}
    \caption{$\nu = 0.499$}
    \label{fig:le_homogeneous_nu_0499_weak_scalability_iterations}
  \end{subfigure}
  \caption{Average iteration count over 50 random initial guesses for the homogeneous linear elasticity discussed in \cref{sec:results_le_homogeneous}. Plots for different values of the Poisson ratio ($\nu$), $H/h = 8$, $\delta = 2h$.}
  \label{fig:le_homogeneous_weak_scalability_iterations}
\end{figure}

\begin{figure}[t]
  \centering
  \begin{subfigure}[b]{0.495\textwidth}
    \centering
    \includegraphics[width=\textwidth]{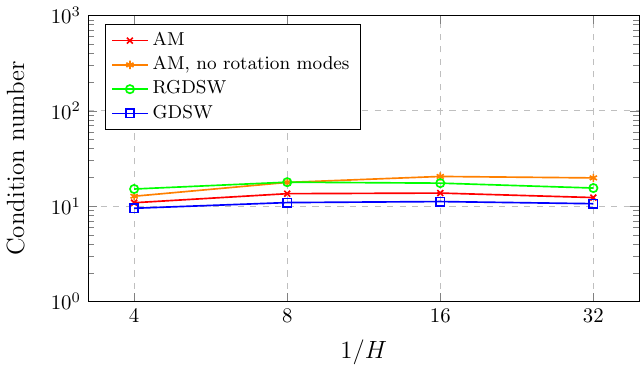}
    \caption{$\nu = 0.3$}
    \label{fig:le_homogeneous_nu_03_weak_scalability_cond_num}
  \end{subfigure}
  \hfill
  \begin{subfigure}[b]{0.495\textwidth}
    \centering
    \includegraphics[width=\textwidth]{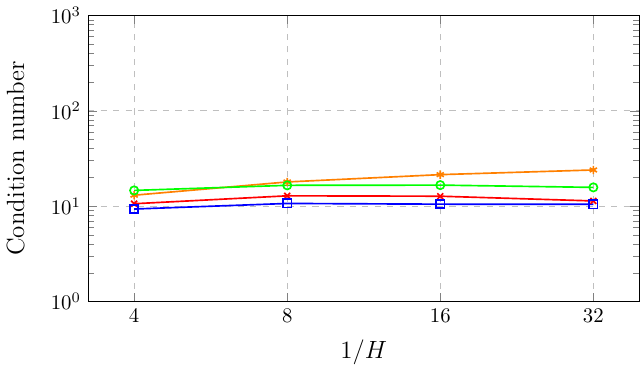}
    \caption{$\nu = 0.4$}
    \label{fig:le_homogeneous_nu_04_weak_scalability_cond_num}
  \end{subfigure}
  \hfill
  \begin{subfigure}[b]{0.495\textwidth}
    \centering
    \includegraphics[width=\textwidth]{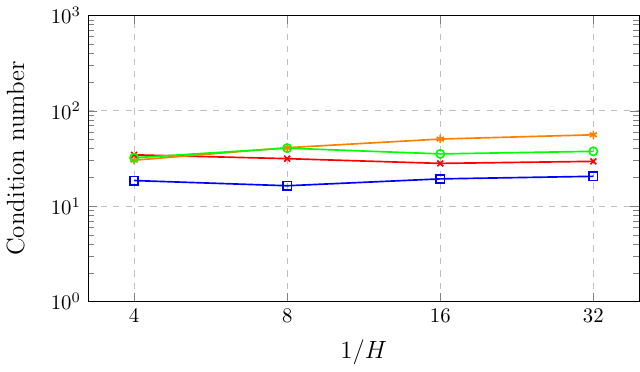}
    \caption{$\nu = 0.49$}
    \label{fig:le_homogeneous_nu_049_weak_scalability_cond_num}
  \end{subfigure}
  \hfill
  \begin{subfigure}[b]{0.495\textwidth}
    \centering
    \includegraphics[width=\textwidth]{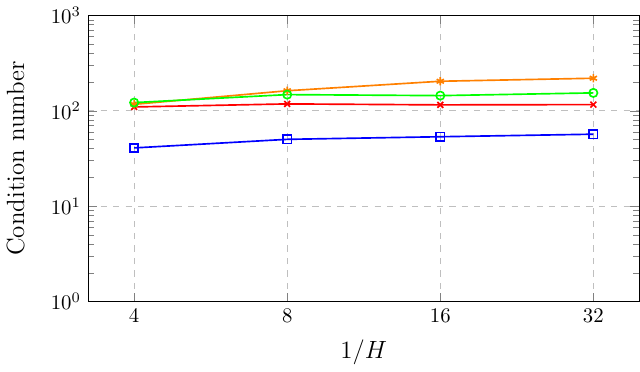}
    \caption{$\nu = 0.499$}
    \label{fig:le_homogeneous_nu_0499_weak_scalability_cond_num}
  \end{subfigure}
  \caption{Condition number estimate for the homogeneous linear elasticity discussed in \cref{sec:results_le_homogeneous}. Plots for different values of the Poisson ratio ($\nu$), $H/h = 8$, $\delta = 2h$.}
  \label{fig:le_homogeneous_weak_scalability_cond_num}
\end{figure}

\begin{table}[t]
  \caption{Number of iterations and condition number estimate for the AM coarse space with (left) and without (right) the rotation modes for different combinations of overlap ($\delta$) and subdomain ($H$) sizes. Boldface values correspond to the case where $H / \delta = 8$. For the examples here, $h = 1/128$.}
  \label{tab:linear_elasticity_homogeneous_parametric_study}

  \centering
  \footnotesize
  \setlength\tabcolsep{0pt}

  \begin{subtable}[b]{0.475\linewidth}
    \begin{tabular*}{\textwidth}{@{\extracolsep{\fill}}*{7}{c} }
      & \multicolumn{3}{c}{Iterations} & \multicolumn{3}{c}{Condition number} \\ 
      \cmidrule{2-4} \cmidrule{5-7}
      \diagbox{$\delta$}{$H$} & $8h$ & $16h$ & $32h$ & $8h$ & $16h$ & $32h$ \\
      \midrule
      $1h$ & \textbf{18} & 21 & 25 & \textbf{12.73} & 22.02 & 34.01 \\
      $2h$ & 18 & \textbf{19} & 22 & 12.80 & \textbf{17.40} & 27.60 \\
      $4h$ & 18 & 18 & \textbf{20} & 12.62 & 13.15 & \textbf{22.05} \\
    \end{tabular*}  
    \caption{AM} \label{tab:linear_elasticity_homogeneous_parametric_study_ams}
  \end{subtable}
  \hfill
  \begin{subtable}[b]{0.475\linewidth}
    \begin{tabular*}{\textwidth}{@{\extracolsep{\fill}}*{7}{c} }
      & \multicolumn{3}{c}{Iterations} & \multicolumn{3}{c}{Condition number} \\ 
      \cmidrule{2-4} \cmidrule{5-7}
      \diagbox{$\delta$}{$H$} & $8h$ & $16h$ & $32h$ & $8h$ & $16h$ & $32h$ \\
      \midrule
      $1h$ & \textbf{25} & 28 & 31 & \textbf{25.83} & 37.04 & 55.15 \\
      $2h$ & 22 & \textbf{25} & 27 & 20.28 & \textbf{27.72} & 39.25 \\
      $4h$ & 21 & 22 & \textbf{21} & 19.26 & 19.99 & \textbf{25.60} \\
    \end{tabular*}
    \caption{AM without rotation modes} \label{tab:linear_elasticity_homogeneous_parametric_study_ams_no_rotations}
  \end{subtable}
\end{table}

\subsection{Linear elasticity: 2D high coefficient inclusions} \label{sec:results_le_high_coeff}

In this subsection, we consider a heterogeneous problem similar to the one discussed in \cref{sec:results_diffusion_high_coeff}. Here, we let the Young modulus be defined as in \cref{fig:vertex_inclusions}, with the Poisson ratio $\nu = 0.3$ and a constant body force term $\mathbf{f} = \mathbf{1}$.

The results on the weak scalability of the AM coarse space are shown in \cref{fig:linear_elasticity_high_coeff_inclusions_weak_scalability} for $H/h = 8$ and $\delta = 2h$. When the entire null space is used for the coarse basis functions, the AM approach is robust with respect to the high coefficient contrast. In this case, the condition number is independent of the contrast. When the rotation modes are omitted, the number of iterations does not increase as sharply as for GDSW and RGDSW, although the condition number is of the same order of magnitude as the contrast.

The spectra of eigenvalues for each coarse space option are presented in \cref{fig:linear_elasticity_high_coeff_inclusions_eigval_spectrum} for $1/H = 16$ and $1/H = 64$. As expected from the condition number computed in \cref{fig:linear_elasticity_high_coeff_inclusions_weak_scalability_cond_num}, the eigenvalues of the AM coarse space with the rotation modes are clustered between $10^{-1}$ and 10, whereas all the other coarse spaces have clusters of unresolved small eigenvalues. However, unlike GDSW and RGDSW, for the AM coarse space without the rotation modes, the cluster of small eigenvalues does not spread across the spectrum as more subdomains are used.

Finally, we have also investigated the impact of the contrast on the conditioning of the system. The results for $1/H = 8$, $1/h = 64$ and $\delta = 2h$ are shown in \cref{tab:linear_elasticity_high_coeff_contrast_study}. The AM approach outperforms GDSW and RGDSW in terms of the number of iterations, but the method becomes robust with respect to the contrast only when the rotation modes are included in the coarse space. Otherwise, the condition number is always on the order of magnitude of the contrast.

\begin{figure}[t]
    \centering
    \begin{subfigure}[b]{0.495\textwidth}
        \includegraphics[width=\textwidth]{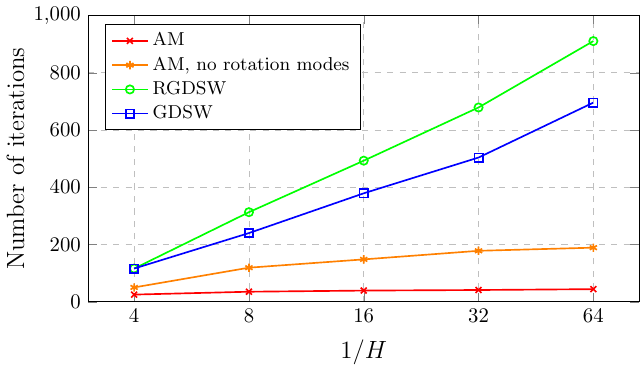}
        \caption{}
        \label{fig:linear_elasticity_high_coeff_inclusions_weak_scalability_iter}
    \end{subfigure}
    \hfil
    \begin{subfigure}[b]{0.495\textwidth}
        \includegraphics[width=\textwidth]{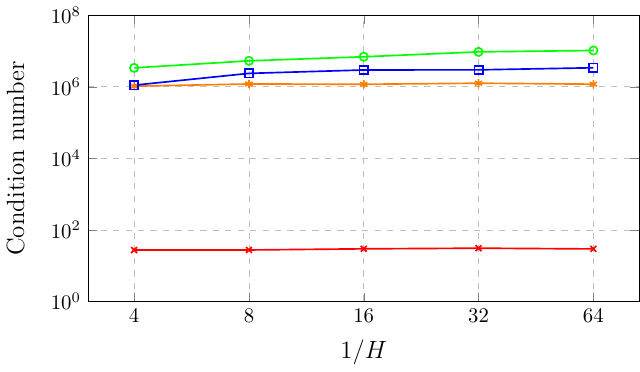}
        \caption{}
        \label{fig:linear_elasticity_high_coeff_inclusions_weak_scalability_cond_num}
    \end{subfigure}
    \caption{Numerical scalability of the AM coarse space compared to the GDSW and RGDSW methods for the coefficient distribution in \cref{fig:vertex_inclusions} applied to the linear elasticity problem. Results for $H/h = 8$ and $\delta = 2h$. \textbf{Left:} number of iterations; \textbf{right:} condition number estimate.}
    \label{fig:linear_elasticity_high_coeff_inclusions_weak_scalability}
\end{figure}

\begin{figure}[t]
  \centering
  \begin{subfigure}[b]{0.495\textwidth}
    \includegraphics[width=\textwidth]{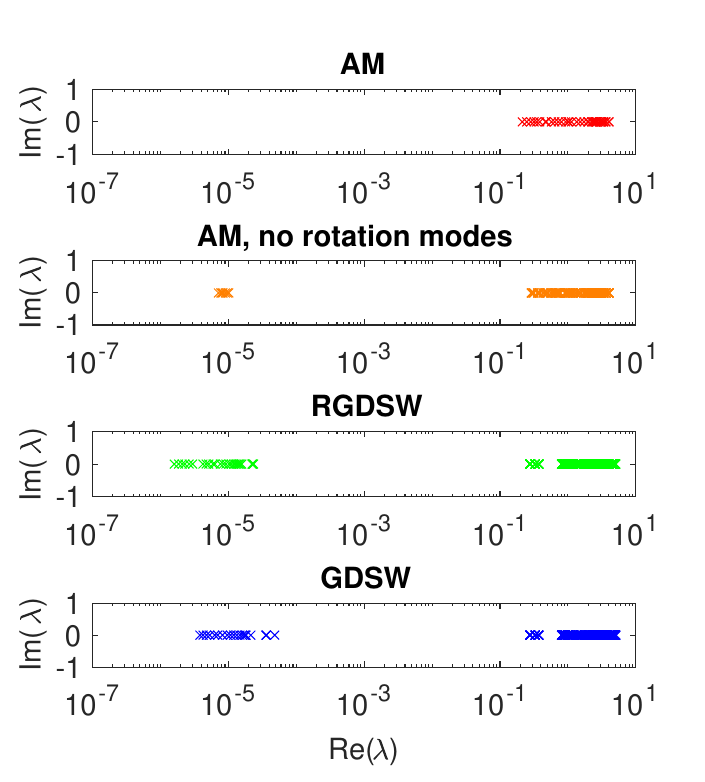}
    \caption{$1/H = 16$}
    \label{fig:linear_elasticity_high_coeff_inclusions_eigval_spectrum_16x16}
  \end{subfigure}
  \hfil
  \begin{subfigure}[b]{0.495\textwidth}
    \includegraphics[width=\textwidth]{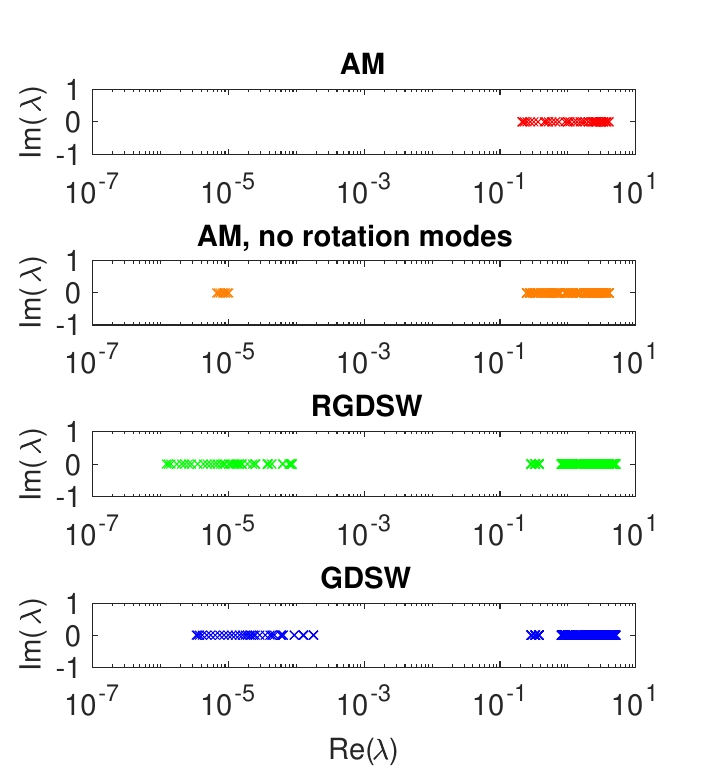}
    \caption{$1/H = 64$}
    \label{fig:linear_elasticity_high_coeff_inclusions_eigval_spectrum_64x64}
  \end{subfigure}
  \caption{Spectra of eigenvalues of the Schwarz operator with each algebraic coarse space option for the coefficient distribution in \cref{fig:vertex_inclusions} applied to the linear elasticity problem. The eigenvalues are computed from the coefficients of the Lanczos algorithm associated with the CG iterations. \textbf{Left:} $1/H = 16$; \textbf{right:} $1/H = 64$.}
  \label{fig:linear_elasticity_high_coeff_inclusions_eigval_spectrum}
\end{figure}

\begin{table}[t]
  \caption{Number of iterations (it.) and condition number estimate ($\kappa$) for each algebraic coarse space option and different Young modulus contrasts. Here, $1/H = 8$, $1/h = 64$, and $\delta = 2h$.}
  \label{tab:linear_elasticity_high_coeff_contrast_study}

  \footnotesize
  \setlength\tabcolsep{0pt}

  \begin{tabular*}{\textwidth}{@{\extracolsep{\fill}}
    >{\RaggedRight}p{2.75cm} 
    *{8}{c} }
    & \multicolumn{2}{c}{$E_{\text{max}} = 1$} & \multicolumn{2}{c}{$E_{\text{max}} = 10^2$} & \multicolumn{2}{c}{$E_{\text{max}} = 10^4$} & \multicolumn{2}{c}{$E_{\text{max}} = 10^6$} \\ 
    \cmidrule{2-3} \cmidrule{4-5} \cmidrule{6-7} \cmidrule{8-9}
    Coarse space & it. & $\kappa$ & it. & $\kappa$ & it. & $\kappa$ & it. & $\kappa$ \\
    \midrule
    AM & 21 & 11.18 & 30 & 25.80 & 33 & 33.42 & 35 & 27.54 \\
    AM, no rotations & 24 & 15.11 & 46 & 104.02 & 82 & $1.11 \cdot 10^4$ & 119 & $1.21 \cdot 10^6$ \\
    RGDSW & 25 & 16.01 & 85 & 456.68 & 201 & $5.30 \cdot 10^4$ & 313 & $5.34 \cdot 10^6$ \\
    GDSW & 21 & 10.51 & 62 & 182.61 & 152 & $1.87 \cdot 10^4$ & 240 & $2.38 \cdot 10^6$ \\
  \end{tabular*}
\end{table}

\subsection{Linear elasticity: 3D unstructured problem} \label{sec:results_le_unstructured}

\begin{figure}[t]
  \centering
  \includegraphics[width=0.5\textwidth]{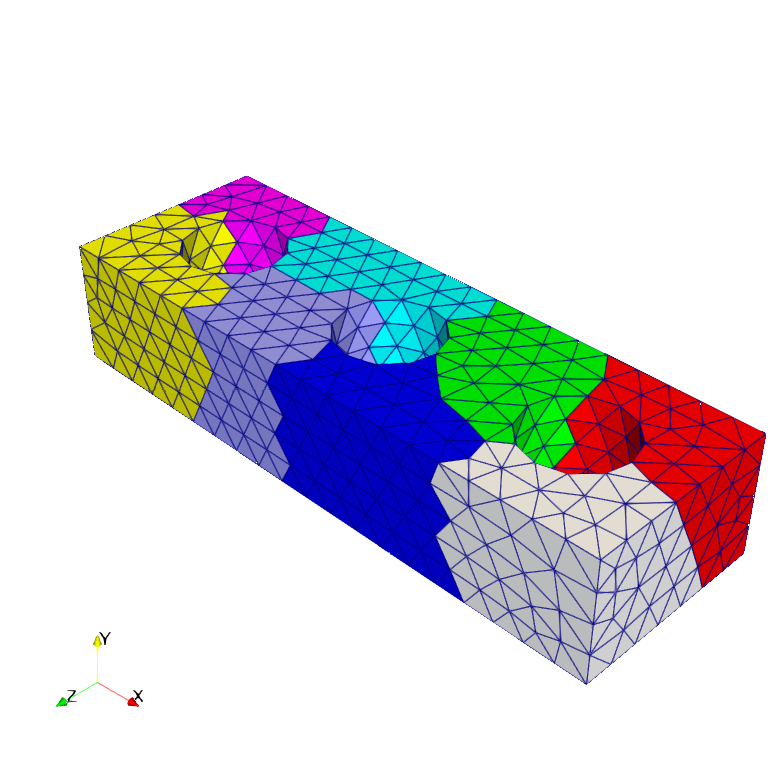}
  \caption{Illustration of the domain used in \cref{sec:results_le_unstructured} and an example of unstructured domain decomposition.}
  \label{fig:linear_elasticity_3d_unstructured_dd}
\end{figure}

In this example, we discuss the performance of the AM coarse space for a 3D unstructured linear elasticity problem. The domain is the rectangular box $[0, 3] \times [0, 0.6] \times [0, 1]$ with three cylindrical holes, as shown in \cref{fig:linear_elasticity_3d_unstructured_dd}. The Young modulus is set to $E = 1$ and the Poisson ratio to $\nu = 0.3$. The body force term is set to $\mathbf{f} = \mathbf{0}$, and the prescribed displacement on the boundary is given by
\begin{equation*}
  \mathbf{u}_{\partial \Omega} =
  \begin{cases}
    (0, 4 x (3 - x) / 9, 0)^{\top}, & \text{on } \partial\Omega_{\text{top}} = \{ (x, y, z) \, | \, y = 0.6 \}, \\
    (0, 0, 0)^{\top}, & \text{elsewhere}.
  \end{cases}
\end{equation*}
The domain was partitioned with METIS \cite{KARYPIS1998METIS} version $5.1.0$ using only the geometric information of the mesh, which does not ensure that the interface of the non-overlapping domain decomposition is algebraically connected.

The weak scalability of the AM coarse space compared to GDSW and RGDSW is shown in \cref{fig:linear_elasticity_3d_unstructured_weak_scalability}. We employed the slab interface extended by one layer of nodes when computing the AM coarse basis functions. Without this correction, the AM edge and face problems are singular, since the interface is not algebraically connected. The results show that the numerical performance of the AM coarse space is comparable to GDSW and RGDSW as in the 2D homogeneous case discussed in \cref{sec:results_le_homogeneous}. Employing the slab entities does not deteriorate the convergence of the two-level preconditioner with the AM coarse space. Moreover, for more than 512 subdomains, we ran out of memory while computing a set of linearly independent GDSW coarse basis functions.

\begin{figure}[t]
    \centering
    \begin{subfigure}[b]{0.495\textwidth}
        \includegraphics[width=\textwidth]{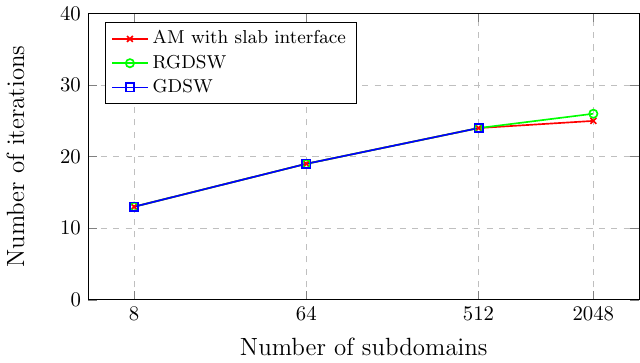}
        \caption{}
        \label{fig:linear_elasticity_3d_unstructured_weak_scalability_iter}
    \end{subfigure}
    \hfil
    \begin{subfigure}[b]{0.495\textwidth}
        \includegraphics[width=\textwidth]{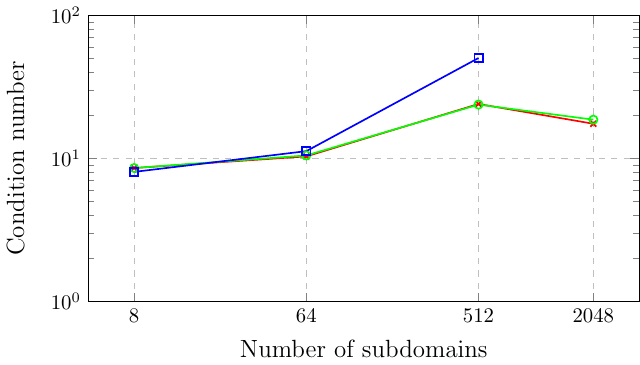}
        \caption{}
        \label{fig:linear_elasticity_3d_unstructured_weak_scalability_cond_num}
    \end{subfigure}
    \caption{Numerical scalability of the AM coarse space compared to the GDSW and RGDSW methods for the unstructured problem discussed in \cref{sec:results_le_unstructured}. Results for $\delta = 2h$. \textbf{Left:} number of iterations; \textbf{right:} condition number estimate.}
    \label{fig:linear_elasticity_3d_unstructured_weak_scalability}
\end{figure}

\subsection{Comparison with AMS} \label{sec:ams_comparison}

In the final numerical experiment, we compare the iterative performance of OAS-2 with the AM coarse space to the AMS preconditioner proposed in \cite{ZHOUTCHELEPI2012TAMS}. AMS is a two-stage method that comprises a local stage aimed at resolving high-frequency error modes, and a global stage given by the solution of the coarse problem defined by the AM coarse basis functions. Convergence can be further improved by successively applying the local stage as a smoother. The main steps of the application of the preconditioner $y^{(k)} = M_{\text{AMS}}^{-1} x^{(k)}$ are summarized in \cref{alg:ams_precond}. It can be interpreted as a multigrid V-cycle with a single coarse level. Based on the computational study conducted in \cite{WANG2014AMS}, we have chosen ILU(0) as the local stage, i.e., $M_{\text{local}}^{-1} = M_{ILU(0)}^{-1}$.

\begin{algorithm}[t]
\caption{AMS preconditioner}
\label{alg:ams_precond}
\begin{algorithmic}
  \STATE{$y^{(k),\, 1} \gets \mathbf{0}$}
  \FOR{$\nu = 1, \dots, n_s$}
    \STATE{$y^{(k),\, \nu + 1} \gets y^{(k),\, \nu} + M^{-1}_{\text{local}} \left( I - A y^{(k)} \right)$}
  \ENDFOR
  \STATE{$r^{(k),\, \nu} \gets x^{(k)} - A y^{(k),\, \nu}$}
  \STATE{$r^{(k),\, \nu}_{c} \gets \Phi^{\top} r^{(k),\, \nu}$}
  \STATE{$\delta^{(k),\, \nu}_{c} \gets A_0^{-1} r^{(k),\, \nu}_{c}$}
  \STATE{$y^{(k)} \gets y^{(k),\, \nu} + \Phi \delta^{(k),\, \nu}_{c}$}
\end{algorithmic}
\end{algorithm}

We consider a heterogeneous scalar diffusion problem on a unit cube with a zero forcing term. The coefficient function is based on 20 realizations of a log-normal distribution. An example is depicted in \cref{fig:geostat_coeff_patchy}. As boundary conditions, we have set $u = 0$ at $x = 0$, $u = 1$ at $x = 1$, and $\nabla u \cdot \mathbf{n} = 0$ elsewhere, where $\mathbf{n}$ is the outward unit normal vector to the boundary. The resulting system of equations is solved by GMRES with a relative stopping criterion $||M^{-1} r^{(k)}||_2 / ||M^{-1} r^{(0)}||_2 < 10^{-5}$, where $M^{-1} r^{(k)}$ is the $k$-th preconditioned residual vector. Moreover, the GMRES iterations are restarted after 50 inner iterations, and, for OAS-2, $\delta = 2h$ and $H/h = 8$.

The iteration counts of OAS-2 with the AM coarse space and of AMS with and without smoothing are shown in \cref{tab:ams_comparison}. On average, the two-level method outperforms AMS without smoothing and is comparable to AMS when additional smoothing is employed. However, OAS-2 with AM is less sensitive to the changes in the coefficient function, as can be observed from the standard deviation for different numbers of subdomains. In \cref{fig:ams_comparison_iter_history}, the iteration histories are shown for a selected realization of the log-normal distribution.

For problems where the two methods perform similarly in terms of the number of iterations, the main advantage of the OAS-2 preconditioner is that the local subdomain problems can be solved in parallel, whereas ILU(0) is inherently sequential. Additionally, the computation of the ILU(0) decomposition scales linearly with the number of non-zero entries in $A$, whereas the factorization of the submatrices $A_i$ depends on the number of dofs in the overlapping subdomains. Naturally, other options for the local stage of \cref{alg:ams_precond} that are better suited for parallel computations can be explored.

The application of the AMS preconditioner using ILU(0) as the local stage can be performed in $O(n + n_V^2)$ operations, where $n$ is the number of nodes in the grid. This cost comprises the application of the smoother and the solution of the coarse problem. As discussed in \cref{sec:comp_analysis}, the computational cost of an application of the two-level OAS preconditioner is $O(Nn_{\Omega_i}^2 + n_{\text{V}}^2)$. Compared to AMS, the only difference lies in the cost of solving the local stage. For OAS, the factor $O(Nn_{\Omega_i}^2)$ is effectively reduced to $O(n_{\Omega_i}^2)$ in a parallel implementation. It might be further reduced to $O(n_{\Omega_i})$ if inexact local subdomain solvers are employed. Furthermore, it does not directly depend on $n$ and, if $n_{\Omega_i}$ is kept constant, the asymptotic computational cost is not affected as the mesh is refined.

\begin{table}[t]
  \caption{Comparison of the iterative performance of AMS against the OAS-2 preconditioner with the AM coarse space. The average, standard deviation, maximum and minimum number of inner GMRES iterations are taken over 20 equiprobable realizations of the log-normal coefficient distribution.}
  \label{tab:ams_comparison}

  \centering
  \footnotesize

  \begin{subtable}{0.75\textwidth}
    \begin{tabular*}{\textwidth}{@{\extracolsep{\fill}}
      >{\RaggedRight}p{2cm} 
      *{5}{l} }
      $1/H$ & Avg. iterations & Std. dev. & Max. & Min.  \\
      \midrule
      4 & 46 & 4.69 & 52 & 37 \\
      8 & 50 & 7.07 & 71 & 43 \\
      16 & 39 & 5.55 & 51 & 30 \\
    \end{tabular*}
    \caption{OAS-2 + AM} \label{tab:ams_comparison_oas2_am}
  \end{subtable}
  \hfill
  \begin{subtable}{0.75\linewidth}
    \begin{tabular*}{\textwidth}{@{\extracolsep{\fill}}
      >{\RaggedRight}p{2cm} 
      *{5}{l} }
      $1/H$ & Avg. iterations & Std. dev. & Max. & Min.  \\
      \midrule
      4 & 232 & 162.54 & 800 & 71 \\
      8 & 92 & 41.27 & 216 & 50 \\
      16 & 47 & 14.74 & 85 & 29 \\
    \end{tabular*}
    \caption{AMS} \label{tab:ams_comparison_ams}
  \end{subtable}
  \hfill
  \begin{subtable}{0.75\linewidth}
    \begin{tabular*}{\textwidth}{@{\extracolsep{\fill}}
      >{\RaggedRight}p{2cm} 
      *{5}{l} }
      $1/H$ & Avg. iterations & Std. dev. & Max. & Min.  \\
      \midrule
      4 & 89 & 49.23 & 268 & 50 \\
      8 & 42 & 10.45 & 76 & 30 \\
      16 & 21 & 2.13 & 25 & 17 \\
    \end{tabular*}
    \caption{AMS with 5 smoothing steps} \label{tab:ams_comparison_ams_w_smoother}
  \end{subtable}
\end{table}

\begin{figure}[t]
  \centering
  \includegraphics[width=0.75\textwidth]{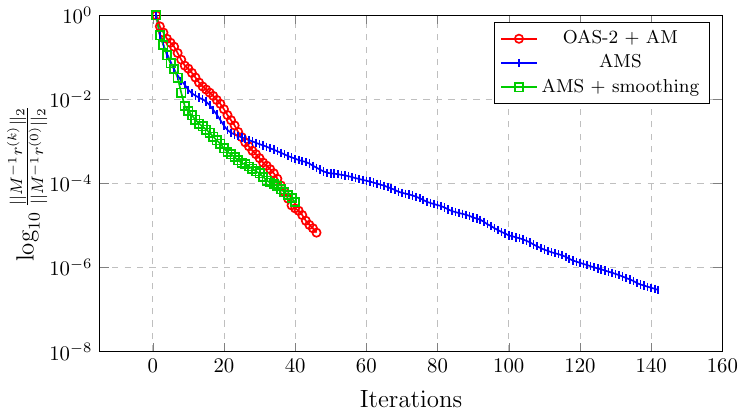}
  \caption{Iteration histories until convergence of OAS-2 with the AM coarse space, AMS, and AMS with smoothing for a realization of the patchy log-normal distribution. The relative $l^2$-norm of the preconditioned residual is shown at each inner GMRES iteration. In this example, $1/H = 8$, $H/h = 8$, and, for OAS-2, $\delta = 2h$.}
  \label{fig:ams_comparison_iter_history}
\end{figure}

\clearpage

\section{Conclusions} \label{sec:conclusions}

In this work, we have introduced an algebraic multiscale (AM) coarse space for the two-level overlapping additive Schwarz method. The coarse basis functions are defined as energy-minimizing extensions of trace values given by the solution of reduced boundary condition problems on the interface of the subdomains. This allows for information about changes in the coefficient function to be incorporated into the functions. In addition, we have also studied the algebraic multiscale formulation as a way to compute an IPOU function that can be used in the RGDSW framework \cite{DOHRMANN2017RGDSW}, thus enabling the coarse space to be extended to linear elasticity problems. Our numerical experiments showed that the AM coarse space outperforms or matches similar algebraic energy-minimizing approaches for the heterogeneous problems studied while using only vertex-centered coarse basis functions. The OAS-2 preconditioner with the AM coarse space is also comparable to the AMS preconditioner in terms of its iterative performance. Furthermore, by computing the interface trace values on a slab entity we enable the use of the AM coarse space for general unstructured domain decompositions while improving its robustness. For linear elasticity, the rotation modes may be omitted while still maintaining good iterative performance. Future work should focus on the parallel implementation of the methods discussed here as part of the Fast and Robust Overlapping Schwarz (FROSch) package \cite{HEINLEIN2020FROSCH}, including CPU performance measurements to fairly compare them with other state-of-the-art preconditioners. Furthermore, in cases where the AM coarse space is not robust with respect to the coefficient, spectral enrichment may be investigated to reinforce robustness by including the problematic eigenmodes into the coarse space.

\bibliographystyle{etna}
\bibliography{bibliography.bib}

\end{document}